\documentclass[11pt]{amsart}
\usepackage{amssymb, enumerate, mdwlist, amsthm, amsmath, bm, graphicx}
\usepackage{xcolor,hyperref}
\hypersetup{
    colorlinks=true,
    citecolor=blue,
    linkcolor=red,
    urlcolor=orange,
}

\evensidemargin\paperwidth
\advance\evensidemargin-\textwidth
\advance\oddsidemargin-1in
\evensidemargin\oddsidemargin

\topmargin\paperheight
\advance\topmargin-\textheight
\advance\topmargin-1in

\theoremstyle{plain}
\newtheorem{theorem}{Theorem}[section]
\theoremstyle{plain}
\newtheorem{proposition}[theorem]{Proposition}
\theoremstyle{plain}
\newtheorem{lemma}[theorem]{Lemma}
\theoremstyle{plain}
\newtheorem{corollary}[theorem]{Corollary}
\theoremstyle{plain}
\newtheorem{problem}[theorem]{Problem}
\theoremstyle{plain}

\theoremstyle{definition}
\newtheorem{definition}[theorem]{Definition}
\theoremstyle{remark}
\newtheorem{remark}[theorem]{Remark}
\theoremstyle{remark}
\newtheorem{example}[theorem]{Example}
\theoremstyle{plain}
\newtheorem{mtheorem}{Theorem}

\theoremstyle{plain}
\newtheorem{mcorollary}[mtheorem]{Corollary}
\theoremstyle{definition}

\title[Avoiding a shape of a system of equations]{Avoiding a shape, and the slice rank method for a system of equations}
\author{Masato Mimura\and Norihide Tokushige}

\address{Masato Mimura\\
Mathematical Institute, Tohoku University, Japan}
\email{m.masato.mimura.m@tohoku.ac.jp}
\address{Norihide Tokushige\\
College of Education, Ryukyu University, Japan}
\email{hide@u-ryukyu.ac.jp}
\date{\today}

\begin{document}

\maketitle

\begin{abstract}
Fix a vector space over a finite field and a system of linear equations. We provide estimates, in terms of the dimension of the vector space, of the maximum of the sizes of subsets of the space that do not admit solutions of the system consisting of more than one point. That from above is derived by slice rank method of Tao; to obtain one from below, we define the notion of `dominant reductions' of the system. Furthermore, by adapting a recent argument of Sauermann, we make an estimation of the maximum of the sizes of subsets that are `W shape'-free, that means, there exist no five distinct points forming two overlapping parallelograms.
\end{abstract}

\section{Introduction}\label{section=introduction}
It is a historical and significant problem in extremal combinatorics to study the maximum of the sizes of subsets in a finite abelian group $G$ that does not contain a specified shape, such as a non-degenerate \textit{three-term arithmetic progression} (hereafter, we abbreviate it as a `$3$-AP'). In this paper, we consider that case where $G=\mathbb{F}_p^n$, where $p$ is a prime number. Inspired by the breakthrough by Croot--Lev--Pach \cite{CrootLevPach}, Ellenberg--Gijswijt \cite{EllenbergGijswijt} obtained an upper bound of the size of a subset $A\subseteq \mathbb{F}_p^n$, $p\geq 3$, that is non-degenerate $3$-AP-free, of the form $\#A\leq (cp)^n$. Here $c<1$ is a universal constant; see \cite[Subsection~4.3]{BCCGNSU} for more details on $c$. A (possibly degenerate) $3$-AP $(x_1,x_2,x_3)$ in this order may be seen as a solution of the linear equation $x_1-2x_2+x_3=0$. In this paper, extending this framework to that of a (finite) \textit{system $\mathcal{T}$ of linear equations} (with each equation having zero-sum of the coefficients); more precisely, we study the maximum of the sizes of $A\subseteq \mathbb{F}_p^n$ that does not admit a `non-trivial' solution of $\mathcal{T}$. A motivating example is the system representing $4$-APs:
\begin{equation*}\label{4AP}
\left\{
\begin{aligned}
x_1-2x_2+x_3&=0,\\
x_2-2x_3+x_4&=0.
\end{aligned}
\right.
\tag{$\mathcal{S}_{4\mathrm{AP}}$}
\end{equation*}
However, as known to experts, at present the \textit{slice rank method} of Tao \cite{Tao} does not provide a non-trivial upper bound in this case. In this paper, we succeed in obtaining non-trivial bounds for the maximum of the sizes with respect to another system of equations, whose solution represents an interesting configuration from geometric aspects. 

\begin{theorem}[Avoiding a `$\mathrm{W}$ shape', outlined version]\label{theorem=Wshape}
Let $\mathcal{S}_{\mathrm{W}}$ be the following system of linear equations in $5$ variables:
\begin{equation*}\label{SW}
\left\{
\begin{aligned}
x_1-x_2-x_3+x_4&=0,\\
x_1-2x_3+x_5&=0.
\end{aligned}
\right.
\tag{$\mathcal{S}_{\mathrm{W}}$}
\end{equation*}
Then, there exist universal constants $0<c_1<c_2<1$ such that for every prime $p\geq 3$, the following holds: For $n\in \mathbb{N}$ sufficiently large depending on $p$, the maximum $\mathrm{ex}^{\sharp}_{\mathrm{W},p}(n)$ of subsets $A\subseteq \mathbb{F}_p^n$ that do not admit a solution $(x_1,x_2,x_3,x_4,x_5)$ of \eqref{SW} $\mathrm{consisting}$ $\mathrm{of}$ $\mathrm{five}$ $\mathrm{distinct}$ $\mathrm{points}$ satisfies that
\[
(c_1p)^n\leq \mathrm{ex}^{\sharp}_{\mathrm{W},p}(n) \leq (c_2p)^n.
\]
\end{theorem}
For sufficiently large $p$, we have that $0.5\leq c_1 <c_2\leq 0.985$; see 
Theorem~\ref{mtheorem=Wshape} for a more precise form. Geometrically, the solution  $(x_1,x_2,x_3,x_4,x_5)$ of the system  \eqref{SW} is the combination shape of  a `parallelogram $(x_1,x_2,x_4,x_3)$' and a `$3$-AP $(x_1,x_3,x_5)$'; it may be regarded as a `\textit{$\mathrm{W}$ shape}'. See figure~\ref{fig5}. (In this paper, we will make a distinction between a system of linear equations with coefficients in $\mathbb{F}_p$ and that with coefficients in $\mathbb{Z}$; for this reason, strictly speaking, we should use $\mathcal{S}_{\mathrm{W}}(p)$ for the system in Theorem~\ref{theorem=Wshape} instead of $\mathcal{S}_{\mathrm{W}}$. We will do it from the next section. See Section~\ref{section=result} for more details.)
\begin{figure}[h]
\begin{center}
\includegraphics[width=5.0cm]{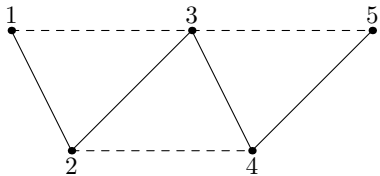}  
\caption{An $\mathcal{S}_{\mathrm{W}}(p)$-shape}\label{fig5}
\end{center}
\end{figure}

The trivial upper bound of $\#A$ above is $\#\mathbb{F}_p^n=p^n$. The estimate from above in Theorem~\ref{theorem=Wshape} shows that for a fixed prime $p\geq 3$, $\frac{\#A}{\#\mathbb{F}_p^n}$ tends to $0$ \textit{exponentially fast} as $n\to \infty$.

Note that there are two possible ways to formulate the notion of  being `\textit{non-trivial}' for a solution of a system $\mathcal{T}$; see Definition~\ref{definition=Erdos} for more details.
\begin{itemize}
 \item Regard a solution consisting of at least two points as a `non-trivial' one; we say that $A\subseteq \mathbb{F}_p^n$ is \textit{strongly $\mathcal{T}$-free} if $A$ does not admit a non-trivial solution in this sense.
 \item Regard a solution consisting of \textit{distinct} points as a `non-trivial' one; we say that $A\subseteq \mathbb{F}_p^n$ is \textit{weakly $\mathcal{T}$-free} if $A$ does not admit a non-trivial solution in that sense.
\end{itemize}
Tao's slice rank method, which reformulates the work of \cite{CrootLevPach} and \cite{EllenbergGijswijt}, applies to the problem of bounding the size of \textit{strongly} $\mathcal{T}$-free sets. However, it is trivial to bound the size of strongly $\mathcal{S}_{\mathrm{W}}$-free sets in $\mathbb{F}_p^n$. Indeed, if $A\subseteq \mathbb{F}_p^n$ admits distinct two points $a,b\in A$, then $(x_1,x_2,x_3,x_4,x_5)=(a,b,a,b,a)$ is a solution violating strong $\mathcal{S}_{\mathrm{W}}$-freeness; it gives that $\#A\leq 1$. On the other hand, if we switch from strong freeness to \textit{weak} freeness, then it becomes a non-trivial problem to study the maximum of the sizes of $\mathcal{S}_{\mathrm{W}}$-free subsets in $\mathbb{F}_p^n$, as Theorem~\ref{theorem=Wshape} indicates. This example exhibits significance of estimating the maximum of the sizes of \textit{weakly} $\mathcal{T}$-free subsets in $\mathbb{F}_p^n$ for a system $\mathcal{T}$; in this formulation, the slice rank method does  \textit{not} immediately apply in general. Note that for the system of \eqref{4AP}, strong freeness and weak freeness are equivalent; the difference described above between these two notions is invisible if we only discuss $k$-APs.

Quite recently, Sauermann \cite{Sauermann} has provided a considerably improved upper bound of the size of a subset $A\subseteq \mathbb{F}_p^n$ that is \textit{weakly} free with respect to the equation in $p$ variables $x_1+x_2+\cdots +x_p=0$ with coefficients in $\mathbb{F}_p$. See Section~\ref{section=Wshape} for more details. The proof of the upper bound in the assertions in Theorem~\ref{theorem=Wshape} will be established in the following two steps.
\begin{enumerate}
  \item[Step~$1$.] For a system $\mathcal{T}$ of linear equations, obtain an upper bound, in the \textit{multicolored version}, of the size of strongly $\mathcal{T}$-free set; see Corollary~\ref{mcorollary=multicolor} and Definition~\ref{definition=multicolored}.
  \item[Step~$2$.] Adapt the argument of Sauermann and deduce the desired upper bound of the size of a \textit{weakly} $\mathcal{T}$-free set from the estimate obtained in Step~$1$.
\end{enumerate}
There is a difficulty in applying the argument in \cite{Sauermann} to the case of our system $\mathcal{S}_{\mathrm{W}}$: When we try to modify one of the key lemmas \cite[Lemma~3.1]{Sauermann} to our case, Lemma~\ref{lemma=extendability}, it is \textit{impossible} to do so in one case. This is exactly the case where the degenerate solution of \eqref{SW} is of the form $(x_1,x_2,x_3,x_4,x_5)=(a,b,a,b,a)$ with $a\ne b$, which appeared above. It occurs because of lack of the full symmetry of the coefficients of the system of equations; it may be seen as a new phenomenon compared with the original case of \cite{Sauermann}. See also Remark~\ref{remark=difference} and Problem~\ref{problem=Sauermann}.

The estimate in Step~$1$ above may be done for a general system of linear equations (with each equation having zero-sum of the coefficients). The proof of the lower bound in the assertions of Theorem~\ref{theorem=Wshape} can be extended to a system that admits a sequence of `\textit{dominant reductions}' that terminates with the empty system of equations in one variable. The main results in the present paper consist of the following three parts. 
\begin{enumerate}[$(i)$]
  \item $[$Theorem~\ref{mtheorem=slicerank} and Corollary~\ref{mcorollary=multicolor}$]$ For a system $\mathcal{T}$ in $r$ variables of $L$ linear equations with coefficients in $\mathbb{F}_p$, we define parameters $(r_1,r_2,L)$, where $r_1+r_2=r$. We obtain a non-trivial upper bound of the size of a subset $A\subseteq \mathbb{F}_p^n$ that is strongly $\mathcal{T}$-free, provided that the parameters $(r_1,r_2,L)$ satisfy `inequality \eqref{bigstar}', which is stated in Section~\ref{section=result}. The proofs here are done by slice rank method.
    \item $[$Theorem~\ref{theorem=dominant}, Theorem~\ref{theorem=dominantsharp} and Theorem~\ref{mtheorem=dominant}$]$ We define a notion of being \textit{dominant} for a system $\mathcal{S}$ of linear equations with \textit{coefficient in $\mathbb{Z}$}. For a system $\mathcal{S}$ that admits a subsystem that is dominant and irreducible (see Definitions~\ref{definition=dominant} and \ref{definition=multiplicity}), we provide a lower bound of the maximum of the sizes of $A\subseteq \mathbb{F}_p^n$ that is strongly/weakly $\mathcal{S}(p)$-free. Here, $\mathcal{S}(p)$ denotes the system of linear equations with coefficients in $\mathbb{F}_p$ obtained by $\mathrm{mod}$ $p$ reduction of $\mathcal{S}$. The proof is inspired by arguments by (Behrend \cite{Behrend} and) Alon \cite[Lemma~17, Corollary~18]{FoxPham}; it employs strict convexity of Euclidean norms (namely, that of the unit sphere in Euclidean spaces). Further, we define a `\textit{dominant reduction}' (Definition~\ref{definition=dominantreduction}), which is an operation on systems of linear equations with coefficients in $\mathbb{Z}$, and generalize the argument above.
  \item $[$Theorem~\ref{mtheorem=Wshape}$]$ As an application of $(i)$ (upper bound) and $(ii)$ (lower bound), we establish Theorem~\ref{theorem=Wshape}. We adapt an argument of Sauermann \cite{Sauermann} to have an upper bound of the size of $A\subseteq \mathbb{F}_p^n$ that is \textit{weakly} `$\mathrm{W}$-shape'-free.
\end{enumerate}
See Section~\ref{section=result} for more details.

\section{Main results}\label{section=result}

Throughout the current paper, let $p$ be a prime number, $q$ be a prime power, and $\mathbb{F}_q$ denote the finite field of order $q$; we mainly use a prime field $\mathbb{F}_p$. Let $\mathbb{N}=\{1,2,\ldots \}$. We use the symbol $\#A$ for the size of a finite set $A$. For $k\in \mathbb{N}$, let $[k]:=\{1,2,\ldots ,k\}$. For a prime $p$, we call a linear equations with coefficients in $\mathbb{F}_p$ simply \textit{an $\mathbb{F}_p$-equation}; we call that with coefficients in $\mathbb{Z}$ simply \textit{a $\mathbb{Z}$-equation}. Unless otherwise stated, we use $a_i^{(l)}(\in \mathbb{F}_p)$ for coefficients of a system of $\mathbb{F}_p$-equations; we use the symbol $b_i^{(l)}(\in \mathbb{Z})$ for those of a system of $\mathbb{Z}$-equations. Similarly, we use $\mathcal{T}$ and $\mathcal{S}(p)$ for a system of $\mathbb{F}_p$-equations, $\mathcal{S}$ for that of $\mathbb{Z}$-equations. We only consider systems of finitely many equations. Let $e=2.71828\ldots$ denote the base of the natural logarithm.

The main subject of the present paper is the maximum of the sizes, viewed as a function on the dimension $n$, of a subset $A\subseteq \mathbb{F}_p^n$ that is strongly/weakly $\mathcal{T}$-free. Here $\mathcal{T}$ is a system of $L$ $\mathbb{F}_p$-equations in $r$ variables:
\begin{equation*}\label{T}
\left\{
\begin{aligned}
a_1^{(1)}x_1+a_2^{(1)}x_2+a_3^{(1)}x_3+\cdots +a_r^{(1)}x_r&=0,\\
a_1^{(2)}x_1+a_2^{(2)}x_2+a_3^{(2)}x_3+\cdots +a_r^{(2)}x_r&=0,\\
\vdots\qquad\qquad \qquad\\
a_1^{(L)}x_1+a_2^{(L)}x_2+a_3^{(L)}x_3+\cdots +a_r^{(L)}x_r&=0.
\end{aligned}
\right.
\tag{$\mathcal{T}$}
\end{equation*}
Here we say that an $\mathbb{F}_p$-equation $a_1x_1+a_2x_2+a_3x_3+\cdots +a_rx_r=0$ is \textit{balanced} if $\sum_{i=1}^ra_i=0$. We say that $\mathcal{T}$ is \textit{balanced} if every equation belonging to $\mathcal{T}$ is balanced. We define the notions of \textit{balanced} $\mathbb{Z}$-equations/systems of $\mathbb{Z}$-equations in the same manner as one above. In this paper, we assume  that systems of $\mathbb{F}_p$- or $\mathbb{Z}$-equations are balanced. We will explicitly write this assumption in the statements of theorems, but even in other parts, we always assume it.

\begin{definition}[Semishapes and shapes]\label{definition=shape}
Let $\mathcal{T}$ be a balanced system of $\mathbb{F}_p$-equations. Let $n\in \mathbb{N}$ and $A_1,A_2,\ldots ,A_r\subseteq \mathbb{F}_p^n$.
\begin{enumerate}[$(1)$]
  \item We say that $(x_1,x_2,\ldots ,x_r)\in A_1\times A_2\times \cdots \times A_r$ is a $\mathcal{T}$-\textit{semishape} in $A_1\times A_2\times \cdots \times A_r$ if it is a solution to $\mathcal{T}$.
  \item We say that a $\mathcal{T}$-semishape $(x_1,x_2,\ldots,x_r)$ is \textit{non-degenerate} if $x_1,\ldots,x_r$ are distinct. A non-degenerate $\mathcal{T}$-semishape is called a \textit{$\mathcal{T}$-shape}.
\end{enumerate}
If $A_1=A_2=\cdots =A_r=A$, then we call a $\mathcal{T}$-(semi)shape in $A_1\times \cdots \times A_r$ that in $A$. 
\end{definition}
For each $k\geq 3$, we use the terminology of \textit{non-degenerate $k$-APs} in the sense above.

Since the system $\mathcal{T}$ is balanced, for $(\emptyset \ne )A\subseteq \mathbb{F}_p^n$ and for $x\in A$, $(x,x,\ldots,x)$ is always a  $\mathcal{T}$-semishape in $A$; we call it a \textit{singleton semishape}.

\begin{definition}[Strong $\mathcal{T}$-freeness and weak $\mathcal{T}$-freeness]\label{definition=Erdos}
Let $\mathcal{T}$ be a balanced system of $\mathbb{F}_p$-equations. Let $n\in \mathbb{N}$.
\begin{enumerate}[$(1)$]
  \item We say that $A\subseteq \mathbb{F}_p^n$ is \textit{strongly $\mathcal{T}$-free} if $\mathcal{T}$-semishapes in $A$ are all singletons. We define $\mathrm{ex}_{\mathcal{T}}(n)\in \mathbb{N}$ by
\[
\mathrm{ex}_{\mathcal{T}}(n):=\max\{\#A:\ A\subseteq \mathbb{F}_p^n,\ \textrm{$A$ is strongly $\mathcal{T}$-free}\}.
\]
  \item We say that $A\subseteq \mathbb{F}_p^n$ is \textit{weakly $\mathcal{T}$-free} if $A$ admits no $\mathcal{T}$-shape. We define $\mathrm{ex}^{\sharp}_{\mathcal{T}}(n)\in \mathbb{N}$ by
\[
\mathrm{ex}^{\sharp}_{\mathcal{T}}(n):=\max\{\#A:\ A\subseteq \mathbb{F}_p^n,\ \textrm{$A$ is weakly $\mathcal{T}$-free} \}.
\]
\end{enumerate}
For a balanced system of $\mathbb{Z}$-equations
\begin{equation*}\label{S}
\left\{
\begin{aligned}
b_1^{(1)}x_1+b_2^{(1)}x_2+b_3^{(1)}x_3+\cdots +b_r^{(1)}x_r&=0,\\
b_1^{(2)}x_1+b_2^{(2)}x_2+b_3^{(2)}x_3+\cdots +b_r^{(2)}x_r&=0,\\
\vdots\qquad\qquad \qquad\\
b_1^{(L)}x_1+b_2^{(L)}x_2+b_3^{(L)}x_3+\cdots +b_r^{(L)}x_r&=0,
\end{aligned}
\right.
\tag{$\mathcal{S}$}
\end{equation*}
we construct, for each prime $p$, a system $\mathcal{S}(p)$ of \textit{$\mathbb{F}_p$-equations} by $\mathrm{mod}$ $p$ reduction of $\mathcal{S}$. We define $\mathrm{ex}_{\mathcal{S}}(n,p)$ and $\mathrm{ex}^{\sharp}_{\mathcal{S}}(n,p)$ respectively by
\[
\mathrm{ex}_{\mathcal{S}}(n,p):=\mathrm{ex}_{\mathcal{S}(p)}(n),\quad \mathrm{ex}^{\sharp}_{\mathcal{S}}(n,p):=\mathrm{ex}^{\sharp}_{\mathcal{S}(p)}(n).
\]
\end{definition}

As we mentioned in the Introduction, we distinguish $\mathbb{F}_p$-equations and $\mathbb{Z}$-equations. This is because the dominant property, defined in Definition~\ref{definition=dominant}, only makes sense for $\mathbb{Z}$-equations.  We use the terminology of solutions, not semishapes, to indicate solutions of a system of $\mathbb{Z}$-equations.

To explain the results mentioned in $(i)$ in the Introduction, we define the concept of a  \textit{multiplicity} of a variable in a system of equations.
 We also need to define \emph{irreducible} systems to state the results 
 mentioned in $(ii)$ in the Introduction. 
 To this end we introduce a hypergraph associated with the system;
 see also Example~\ref{example=irreducible}.
\begin{definition}[Hypergraph associated with $\mathcal{T}$, multiplicities, simple variables]\label{definition=multiplicity}
Let \eqref{T} be a balanced system of $L$ $\mathbb{F}_p$-equations in $r$ variables
\begin{equation*}
\left\{
\begin{aligned}
a_1^{(1)}x_1+a_2^{(1)}x_2+a_3^{(1)}x_3+\cdots +a_r^{(1)}x_r&=0,\\
a_1^{(2)}x_1+a_2^{(2)}x_2+a_3^{(2)}x_3+\cdots +a_r^{(2)}x_r&=0,\\
\vdots\qquad\qquad \qquad\\
a_1^{(L)}x_1+a_2^{(L)}x_2+a_3^{(L)}x_3+\cdots +a_r^{(L)}x_r&=0.
\end{aligned}
\right.
\end{equation*}
\begin{enumerate}[$(1)$]
 \item We define a \textit{hypergraph} (multi-hypergraph) $\mathcal{H}_{\mathcal{T}}=(J,H)=(J_{\mathcal{T}},H_{\mathcal{T}})$ associated with $\mathcal{T}$ in the following manner: For each $l\in [L]$, set an edge $e_l:=\{i\in [r]:a_i^{(l)}\ne 0\}$. Define the multiset of edges as $H=H_{\mathcal{T}}=\{e_l:l\in [L]\}$. Set the set of vertices $J=J_{\mathcal{T}}$ as $J:=\bigcup_{l\in [L]}e_l$.
  \item We say that the system $\mathcal{T}$ is \textit{irreducible} if the hypergraph $\mathcal{H}_{\mathcal{T}}=(J,H)$ satisfies $J=[r]$ and that it is connected. Here we say that $\mathcal{H}_{\mathcal{T}}$ is \textit{connected} if for every $\{i,j\}\in \binom{J}{2}$, there exists a sequence of vertices $i=v_1,v_2,\ldots ,v_k=j$ such that for each $k'\in [k-1]$, there exists $l\in [L]$ with $\{v_{k'},v_{k'+1}\}\subseteq e_l$.
  \item Define the \textit{multiplicity} $m_i$ of a variable $x_i$, $i\in [r]$, in $\mathcal{T}$ by
\[
m_i:=\#\{l\in [L]:i\in e_l\}\quad (=\#\{l\in [L]: a_i^{(l)}\ne 0\});
\]
in other words, $m_i$ is the degree of vertex
 $i\in J$ in the (multi-)hypergraph $\mathcal{H}_{\mathcal{T}}$.
We say that a variable $x_i$ is \textit{simple} if $m_i=1$; otherwise, $x_i$ is said to be \textit{non-simple}.
  \item Define the following two parameters for $\mathcal{T}$:
\[
r_1:=\#\{i\in [r]: m_i=1\},\quad 
r_2:=\#\{i\in [r]: m_i\geq 2\}.
\]
Namely, $r_1$ is the number of simple variables; $r_2$ is the number of variables with multuiplicity at least $2$.
  \item The \textit{maximal multiplicity} $m_{\mathcal{T}}$ of $\mathcal{T}$ is defined as  $m_{\mathcal{T}}:=\max\{m_i:i\in [r]\}$; of course, it is the maximum  degree of the (multi-)hypergraph $\mathcal{H}_{\mathcal{T}}$.
\end{enumerate}
\end{definition}

We usually assume that our balanced system $\mathcal{T}$ is in addition irreducible. In that case, we have that $r_2$ equals the number of non-simple variables, and that $r_1+r_2=r$.

\begin{remark}
For a system $\mathcal{S}$ of $\mathbb{Z}$-equations, we define the hypergraph $\mathcal{H}_{\mathcal{S}}$, irreducibility, the parameters $r_1$ and $r_2$, If a prime is strictly bigger than the maximum of the absolute values of the coefficients in $\mathcal{S}$, then these  give exactly the same pieces of information as those associated with the system $\mathcal{S}(p)$ constructed by $\mathrm{mod}$ $p$ reduction.
\end{remark}

\begin{example}\label{example=irreducible}
Consider the following system of $\mathbb{Z}$-equations in $9$ variables:
 \begin{equation*}\label{S1}
 \left\{
 \begin{aligned}
 x_1+x_2-x_3-x_4 &=0,\\
 x_5+x_6-2x_7&=0,\\
 x_5+x_7+x_8-3x_9&=0,\\
 x_4-2x_5+x_8&=0.
 \end{aligned}
 \right.
 \tag{$\mathcal{S}_1$}
 \end{equation*}
 Then the hypergraph $\mathcal{H}_{\mathcal{S}_1}=(J_{\mathcal{S}_1},H_{\mathcal{S}_1})$ is the 
 following: $J_{\mathcal{S}_1}=[9]$, and $H_{\mathcal{S}_1}$ is the (multi)set $\{e_1,e_2,e_3,e_4\}$,
 where $e_1=\{1,2,3,4\}$, $e_2=\{5,6,7\}$, $e_3=\{5,7,8,9\}$ and $e_4=\{4,5,8\}$.
 The simple variables of $\mathcal{S}$ are $x_1,x_2,x_3,x_6,x_9$, and the parameters for $\mathcal{S}_1$ are $(r_1,r_2,L,m_{\mathcal{S}_1})=(5,4,4,3)$. (Indeed, $5\in [9]$ is contained in $e_2$, $e_3$ and $e_4$.) Since $J_{\mathcal{S}_1}=[9]$ and 
 $\mathcal{H}_{\mathcal{S}_1}$ is connected, $\mathcal{S}_1$ is irreducible. However, if we construct 
 a subsystem consisting of the first, second and third equations,
 \begin{equation*}\label{S1'}
 \left\{
 \begin{aligned}
 x_1+x_2-x_3-x_4 &=0,\\
 x_5+x_6-2x_7&=0,\\
x_5+x_7+x_8-3x_9&=0,\\
 \end{aligned}
 \right.
 \tag{$\mathcal{S}_1'$}
 \end{equation*}
 then this $\mathcal{S}_1'$ is no longer irreducible. Indeed, despite that $J_{\mathcal{S}_1'}=[9]$, 
 the connected components of $\mathcal{H}_{\mathcal{S}_1'}$ consist of $\{1,2,3,4\}$ and 
 $\{5,6,7,8,9\}$.
 Hence this system $\mathcal{S}_1'$ may be seen as the union of two `independent' systems, one is in 
 variables $x_1,x_2,x_3,x_4$, and the other is in variables $x_5,x_6,x_7,x_8,x_9$.
 Similarly, if we construct another subsystem $\mathcal{S}_1''$ by picking up the first, second and 
 fourth equations, then $\mathcal{S}_1''$ is not irreducible because $J_{\mathcal{S}_1''}=[8]$; the 
 variable $x_9$ is missing in the subsystem $\mathcal{S}_1''$.

These examples indicate that a subsystem of irreducible system is \textit{not necessarily} irreducible. Hence, although we usually assume that an original system of equations is irreducible, we may sometimes encounter with subsystems that are not irreducible. More precisely, in Definition~\ref{definition=dominantreduction}, we utilize the connected components of $\mathcal{H}_{\mathcal{S}'}$ for  a subsystem $\mathcal{S}'$ of $\mathcal{S}$, which is \textit{not} irreducible in general.
\end{example}

\begin{mtheorem}[`Inequality $(\bigstar )$' and upper bounds of $\mathrm{ex}_{\mathcal{T}}(n)$]\label{mtheorem=slicerank}
 
Let $p$ be a prime. Let $\mathcal{T}$ be a balanced and irreducible system of $L$ $\mathbb{F}_p$-equations in $r$ variables. Let $r_1$ and $r_2$ be the parameters for $\mathcal{T}$ defined in Definition~$\ref{definition=multiplicity}$. Assume that the parameters $(r_1,r_2,L)$ satisfy the following \textit{`inequality'}
\[\label{bigstar}
\frac{1}{2}r_1+\frac{1}{e}r_2>L.\tag{$\bigstar$}
\]

Then, there exists $c_{\mathcal{T}}(p)<1$ such that for all $n\in \mathbb{N}$, it holds that
\[
\mathrm{ex}_{\mathcal{T}}(n)\leq (c_{\mathcal{T}}(p)\cdot p)^n.
\]
Furthermore, we have that
\[
c_{\mathcal{T}}(p) \leq \frac{\tilde{C}_{(r_1,r_2,L,m_{\mathcal{T}})}(p)}{p},
\]
where the constant $\tilde{C}_{(r_1,r_2,L,m_{\mathcal{T}})}(p)$ is defined in Definition~$\ref{definition=Lambda}$$(2)$ below.\end{mtheorem}

See Proposition~\ref{proposition=slicerank} for more details of the upper bound of $\mathrm{ex}_{\mathcal{T}}(n)$. It may be possible to improve the coefficient $\frac{1}{e}$ in `inequality \eqref{bigstar}' if $m_{\mathcal{T}}$ is bounded from above. However, in the present paper, we do not perform  a further optimization for each $m_{\mathcal{T}}$. For an irreducible system of a \textit{single} equation in $r$ variables, we always have that $(r_1,r_2,L)=(r,0,1)$. Hence, we may see `inequality \eqref{bigstar}' only after we consider a \textit{system} of equations. 

The parameters associated with the system \eqref{4AP} in the Introduction, which characterizes a $4$-AP, are $(r_1,r_2,L)=(2,2,2)$; unfortunately, they do \textit{not} fulfill `inequality \eqref{bigstar}'. However, there exists other systems $\mathcal{S}$ of $\mathbb{Z}$-inequalities for which Theorem~\ref{mtheorem=slicerank} provides a non-trivial upper bound of $\mathrm{ex}_{\mathcal{S}}(n,p)$; some of them appear naturally from a geometric point of view, For instance, the system  \eqref{SW} representing a `$\mathrm{W}$ shape', 
\begin{equation*}
\left\{
\begin{aligned}
x_1-x_2-x_3+x_4&=0,\\
x_1-2x_3+x_5&=0
\end{aligned}
\right.
\end{equation*}
is such an example if we switch from $\mathrm{ex}_{\mathcal{S}_{\mathrm{W}}}(n,p)$ to $\mathrm{ex}^{\sharp}_{\mathcal{S}_{\mathrm{W}}}(n,p)$; compare with arguments in the Introduction.

The following is the definition of the constant $\tilde{C}_{(r_1,r_2,L,m)}(d)$ which appears in Theorem~\ref{mtheorem=slicerank}. See Lemma~\ref{lemma=sum} for properties of the constant $\Lambda_{m,\alpha,h}$, and see Proposition~\ref{proposition=slicerank} for the background of the constant $\tilde{C}_{(r_1,r_2,L,m)}(d)$.

\begin{definition}[The constants $\Lambda_{m,\alpha,h}$ and $\tilde{C}_{(r_1,r_2,L,m)}(d)$]\label{definition=Lambda}
\begin{enumerate}[$(1)$]
  \item Let $m \in \mathbb{N}$, $\alpha\in \mathbb{R}_{>0}$ and $h\in \mathbb{N}$. Define a function
\begin{align*}
G_{m,\alpha,h}\colon (0,1]\to \mathbb{R};\quad u\mapsto u^{-\alpha h}(1+u+u^2+u^3+\cdots +u^{mh}).
\end{align*}
Define $\Lambda_{m,\alpha,h}$ to be the minimum of $G_{m,\alpha,h}$ over $(0,1]$. If $\alpha=0$, then define $\Lambda_{m,0,h}:=1$.
  \item Let $r_1,r_2\in \mathbb{Z}_{\geq 0}$ with $(r_1,r_2)\ne (0,0)$ and let $L,m\in \mathbb{N}$. For $d\geq 2$, define the following constant $\tilde{C}_{(r_1,r_2,L,m)}(d)$ by
\[
\tilde{C}_{(r_1,r_2,L,m)}(d):=\inf\{\max\{\Lambda_{1,\alpha,d-1},\Lambda_{m,\beta,d-1}\}:\ r_1\alpha+r_2\beta=L,\ \alpha\geq 0,\ \beta\geq 0\}.
\]
\end{enumerate}
\end{definition}

\begin{remark}
In the literature on the slice rank method, $\Lambda_{1,\frac{1}{k},p-1}$ ($k\geq 3$) in our paper is sometimes expressed as $\Gamma_{p,k}$. In the study of a \textit{system} of equations, it seems that $\Lambda_{m,\alpha,h}$ is more suited; this is the reason why we switch the symbol `$\Gamma$' to a new symbol `$\Lambda$'.
\end{remark}

Next, we will state the result in point $(ii)$ as in the Introduction. The main result here, Theorem~\ref{mtheorem=dominant}, is somewhat technical (which involves `\textit{dominant reductions}' as in Definition~\ref{definition=dominantreduction}); hence we only exhibit the most basic case of it in this section. The notion of being `dominant' for a $\mathbb{Z}$-equation is defined as follows.

\begin{definition}[Dominant $\mathbb{Z}$-equations]\label{definition=dominant}
\begin{enumerate}[$(1)$]
\item A balanced $\mathbb{Z}$-equation $b_1x_1+b_2x_2+b_3x_3+\cdots +b_rx_r=0$ is said to be \textit{dominant} if either of the following is satisfied:
\begin{itemize}
  \item There exists $j\in [r]$ such that $b_j>0$ and that for all $i\in [r]\setminus \{j\}$, $b_i\leq 0$ holds,
  \item there exists $j\in [r]$ such that $b_j<0$ and that for all $i\in [r]\setminus \{j\}$, $b_i\geq 0$ holds.
\end{itemize}
We call the index $j$ above a \textit{dominant index}, and $|b_j|(\in \mathbb{N})$ the \textit{dominant coefficient} of the equation.
\item We say a balanced system $\mathcal{S}$ of $\mathbb{Z}$-equations is \textit{dominant} if every equation belonging to $\mathcal{S}$ is dominant. For a balanced and dominant system $\mathcal{S}$, we define the \textit{dominant coefficient} of $\mathcal{S}$ as the maximum of the dominant coefficients of the equations in $\mathcal{S}$.
\end{enumerate}
\end{definition}

A dominant index of a dominant equation is uniquely determined unless it is of the form $bx_i-bx_j=0$, where $b\in \mathbb{N}$ and  $\{i,j\}\in \binom{[r]}{2}$. An equation of the form above may be seen as an `uninteresting' dominant equation. However, it might appear in the process of dominant reductions; compare with Lemma~\ref{lemma=dominantreduction}. The dominant coefficient of a dominant equation is uniquely determined. A \textit{subsystem} $\mathcal{S}'$ of a system $\mathcal{S}$ of equations in $r$ variables is a nonempty subcollection of equations. We say that a subsystem $\mathcal{S}'$ is \textit{irreducible} if $\mathcal{S}'$, viewed as a system of equations in the original $r$ variables, is irreducible; recall also Example~\ref{example=irreducible}. Note that irreduciblity of a  system of $\mathbb{Z}$-equations \textit{does not necessarily} pass to a subsystem, while the dominant property does.
\begin{theorem}[Lower bound of $\mathrm{ex}_{\mathcal{S}}(n,p)$ for $\mathcal{S}$ admitting a dominant and irreducible subsystem]\label{theorem=dominant}
Let $\mathcal{S}$ be a balanced system of $\mathbb{Z}$-equations. Assume that there exists a subsystem $\mathcal{S}'$ of $\mathcal{S}$ that is dominant and irreducible. Fix such an $\mathcal{S}'$, and let $b_{\mathcal{S}'}$ be the dominant coefficient of  $\mathcal{S}'$.

Then we have the following for every prime $p$ with $p> b_{\mathcal{S}'}$:``For every $\epsilon>0$, there exists $n_0=n_0(p,\epsilon)$ such that for all $n\geq n_0(p,\epsilon)$,
\[
\mathrm{ex}_{\mathcal{S}}(n,p)\geq \left((1-\epsilon)\left\lfloor \frac{p+b_{\mathcal{S}'}-1}{b_{\mathcal{S}'}}\right\rfloor\right)^n
\]
holds true.'' In particular, if moreover $b_{\mathcal{S}'}\geq 2$, then for sufficiently large $n$ depending on $p$, we have that
\[
\mathrm{ex}_{\mathcal{S}}(n,p)\geq \left(\frac{p}{b_{\mathcal{S}'}}\right)^n.
\]
\end{theorem}

By introducing the operation `\textit{dominant reduction}' on balanced systems of $\mathbb{Z}$-equations, we obtain a refinement of Theorem~\ref{theorem=dominant}; see Definition~\ref{definition=dominantreduction} and Theorem~\ref{mtheorem=dominant} for more details. This operation of dominant reductions appears only after we broaden the framework from a single equation to a system of equations; compare with Problem~\ref{problem=dominant}.

Finally, we present the precise form of the upper bound part of the assertions in Theorem~\ref{theorem=Wshape}, which is the main part of $(iii)$ as in the Introduction.

\begin{mtheorem}[Avoiding a `$\mathrm{W}$ shape']\label{mtheorem=Wshape}
Let $p\geq 3$ be a prime. Let $\mathcal{T}_{\mathrm{W},p}$ be the following system of $\mathbb{F}_p$-equations in $5$ variables:
\begin{equation*}\label{TWp}
\left\{
\begin{aligned}
x_1-x_2-x_3+x_4&=0,\\
x_1-2x_3+x_5&=0.
\end{aligned}
\right.
\tag{$\mathcal{T}_{\mathrm{W},p}$}
\end{equation*}
In other words, $\mathcal{T}_{\mathrm{W},p}=\mathcal{S}_{\mathrm{W}}(p)$. Then, for every $p\geq 3$, there exists $C_{\mathrm{W}}(p)<p$ such that for every  $n\in \mathbb{N}$, it holds that
\[
\mathrm{ex}^{\sharp}_{\mathcal{T}_{\mathrm{W},p}}(n)\leq  7\left(\sqrt{C_{\mathrm{W}}(p)\cdot p}\right)^n.
\]
Furthermore, we may have that $C_{\mathrm{W}}(p)=\tilde{C}_{(3,2,2,2)}(p)$, namely,
\[
C_{\mathrm{W}}(p)=\inf\{\max\{\Lambda_{1,\alpha,p-1},\Lambda_{2,\beta,p-1}\}:\ 3\alpha+2\beta=2,\ \alpha\geq 0,\ \beta\geq 0\}.
\]
\end{mtheorem}

The constant $C_{\mathrm{W}}(p)$ above satisfies for sufficiently large $p$,
\[
C_{\mathrm{W}}(p)<0.97p \quad (<(0.985)^2p);
\]
see Remark~\ref{remark=lessthanpW}. The conclusion of Theorem~\ref{mtheorem=Wshape} is rephrased in terms of the system $\mathcal{S}_{\mathrm{W}}$ of $\mathbb{Z}$-equations as 
\[
\mathrm{ex}^{\sharp}_{\mathcal{S}_{\mathrm{W}}}(n,p)\leq  7\left(\sqrt{C_{\mathrm{W}}(p)\cdot p}\right)^n. 
\]
On the lower bound of $\mathrm{ex}^{\sharp}_{\mathcal{S}_{\mathrm{W}}}(n,p)$, see Remark~\ref{remark=lowerW}.

\ 

\noindent
\textbf{Organization of the present paper:} Section~\ref{section=slicerankpreliminary} is for preliminaries on the slice rank method. In Section~\ref{section=slicerank}, we prove Theorem~\ref{mtheorem=slicerank} by showing a precise version of the expression of the upper bound, which appears in Proposition~\ref{proposition=slicerank}. We verify the multicolored version, Corollary~\ref{mcorollary=multicolor}, of Theorem~\ref{mtheorem=slicerank}; it will be employed in the proof of Theorem~\ref{mtheorem=Wshape}. In Section~\ref{section=dominant}, we first prove Theorem~\ref{theorem=dominant}. Then we proceed to the formulation of dominant reductions, and generalize Theorem~\ref{theorem=dominant} to Theorem~\ref{mtheorem=dominant}. Section~\ref{section=Wshape} is devoted to the proofs of Theorem~\ref{mtheorem=Wshape} and Theorem~\ref{theorem=Wshape}. In Section~\ref{section=problems}, we propose two further problems.

\section{Preliminaries on the slice rank method}\label{section=slicerankpreliminary}

Here we briefly recall basics on the slice rank method; see \cite{Tao} and \cite{BCCGNSU} for more details. In his blog post, Tao \cite{Tao} introduced the notion of \textit{slice ranks} as a generalization of that of tensor ranks. In this section, let $F$ be a field and $X$ be a non-empty finite set. Let $r\in \mathbb{N}$.

\begin{definition}[Slice rank, \cite{Tao}]\label{definition=slicerank}
\begin{enumerate}[$(1)$]
 \item A function $g\colon X^r\to F$ is called a \textit{slice function} if there exists an index $i\in [r]$ such that the function $g$ is factorized as $g(x_1,\ldots ,x_r)$$=g_1(x_i)g_2(x_1,x_2,\ldots ,x_{i-1},x_{i+1},\ldots ,x_r)$, as the product of a one-variable function and a function in the rest $(r-1)$ variables.
 \item The \textit{slice rank} $\mathrm{sr}(f)$ of a function $f\colon X^r\to F$ is the minimum of $R$ over all expressions $f=\sum_{k=1}^Rg_k$ of $f$ as the sum of slice functions $g_1,g_2,\ldots ,g_R$.
\end{enumerate}
\end{definition}
The tensor rank of $f$ above is the minimum of $R$ over all expressions of $f$ as the sum of functions $g_1,\ldots ,g_R$ of the from $g_i=g_{i,1}(x_1)g_{i,2}(x_2)\cdots g_{i,r}(x_r)$. It is not too difficult to verify that $\mathrm{sr}(f)\leq \#X$ in general; see \cite{Tao} and \cite[Subsection~4.1]{BCCGNSU}.  Note that the slice rank is subadditive, that means for $f_1,f_2\colon X^r\to F$, it holds that $\mathrm{sr}(f_1+f_2)\leq \mathrm{sr}(f_1)+\mathrm{sr}(f_2)$.

The following proposition explains utility of the slice rank; this is a special case of \cite[Lemma~$1$]{Tao}.

\begin{proposition}[\cite{Tao}]\label{proposition=Tao}
Assume that $f\colon X^r\to F$ satisfies the condition
\[\label{ast}
f(x_1,x_2,\ldots ,x_r)\ne 0 \quad \Longleftrightarrow \quad x_1=x_2=\cdots =x_r. \tag{$\ast$}
\]
Then we have that $\mathrm{sr}(f)=\#X$.
\end{proposition}

\begin{proof}
See \cite[Lemma~1]{Tao} or \cite[Lemma~4.7]{BCCGNSU}.
\end{proof}

Recall from Definition~\ref{definition=Lambda} the definition of $\Lambda_{m,\alpha,h}$.

\begin{lemma}\label{lemma=sum}
Let $n\in \mathbb{N}$. Let $m\in \mathbb{N}$, $\alpha\in \mathbb{R}_{\geq 0}$ and $h\in \mathbb{N}$. Define the set of indices by
\[
\Theta_{m,\alpha,h}(n):=\{(\theta_1,\theta_2,\ldots ,\theta_n)\in \{0,1,\ldots ,mh\}^n: \theta_1+\theta_2+\cdots +\theta_n\leq \alpha hn\}.
\]
Then we have that $\#\Theta_{m,\alpha,h}(n)\leq (\Lambda_{m,\alpha,h})^n$.
\end{lemma}

\begin{proof}
If $\alpha=0$, then the inequality holds trivially because $1\leq 1$. In what follows, we assume that $\alpha>0$.

Given an $n$-tuple $(\theta_1,\theta_2,\ldots ,\theta_n)$ as above, we set for $i\in \{0,1,\ldots ,mh\}$, $j_i:=\#\{l\in [n]:\theta_l=i\}$. Then there is a one-to-one correpsondence between an $(mh+1)$-tuple $(j_0,j_1,\ldots ,j_{mh})$ of non-negative integers with the following two conditions
\begin{equation*}
\left\{
\begin{aligned}
j_0+j_1+j_2+j_3+\cdots+ j_{mh}&=n,\\
j_1+2j_2+3j_3+\cdots +(mh)j_{m}&\leq \alpha hn,
\end{aligned}
\right.
\end{equation*}
and a $\textit{multi}$set $\{\theta_1,\theta_2,\ldots ,\theta_n\}$, where we care the multiplicities, but we do not take the orders into account. This correspondence provides us with the following equality which involves multinomial coefficients:
\[
\#\Theta_{m,\alpha,h}(n)=\sum_{(j_0,j_1,\ldots ,j_{mh})}\binom{n}{j_0,j_1,\ldots ,j_{mh}}.
\]
Here, in the sum on the right-hand side of the equality above, $(j_0,\ldots ,j_{mh})$ runs over all tuple of non-negative integers satisfying the two conditions above. We write $J_{m,\alpha,h}(n)$ for the set of all tuples of non-negative integers $(j_0,\ldots ,j_{mh})$ fulfilling these two conditions. 

For $u\in (0,1]$, we have that
\begin{align*}
(G_{m,\alpha,h}(u))^n&=u^{-\alpha hn}(1+u+u^2+\cdots +u^{mh})^n\\
&=\sum_{(j_0,\ldots ,j_{mh})}\binom{n}{j_0,j_1,\ldots ,j_{mh}}u^{j_1+2j_2+\cdots +(mh)j_{mh}-\alpha hn}.
\end{align*}
Here in the sum on the very below side of the equalities above, the indices run over all non-negative integers with $j_0+j_1+\cdots+j_{mh}=n$. By removing indices with $j_1+2j_2+3j_3+\cdots +(mh)j_{mh}> \alpha hn$ from the sum above, we obtain that 
\begin{align*}
(G_{m,\alpha,h}(u))^n\geq \sum_{(j_0,\ldots ,j_{mh})\in J_{m,\alpha,h}(n)}\binom{n}{j_0,j_1,\ldots ,j_{mh}}u^{j_1+2j_2+\cdots +(mh)j_{mh}-\alpha hn}.
\end{align*}
Since $u\in (0,1]$, for each term in the sum on the right-hand side of the inequality above, it holds that $u^{j_1+2j_2+\cdots +(mh) j_{mh}-\alpha hn}\geq 1$. Therefore, we have that for every $u\in (0,1]$,
\begin{align*}
(G_{m,\alpha,h}(u))^n&\geq \sum_{(j_0,\ldots ,j_{mh})\in J_{m,\alpha,h}(n)}\binom{n}{j_0,j_1,\ldots ,j_{mh}}\\
&=\# \Theta_{m,\alpha,h}(n).
\end{align*}
We obtain the conclusion by taking the minimum of $G_{m,\alpha,h}(u)$ over $u\in (0,1]$.
\end{proof}

The following lemma is the background of `inequality \eqref{bigstar}', which appears in Theorem~\ref{mtheorem=slicerank}.

\begin{lemma}[Comparison between $\Lambda_{m,\alpha,d-1}$ and $d$]\label{lemma=lessthanp}
Let $d\geq 2$ be an integer.
\begin{enumerate}[$(1)$]
  \item Let $\alpha \in (0,\frac{1}{2})$. Then, we have 
\[
\frac{\Lambda_{1,\alpha,d-1}}{d}<1.
\]
Furthermore, $\frac{1}{d}\Lambda_{1,\alpha,d-1}$ is decreasing in $d\geq 2$. 
  \item Let $\beta\in (0,\frac{1}{e}]$. Then,  for every $m\in \mathbb{N}$, there exists $d_0(m,\beta)\geq 2$ such that for every $d\geq d_0(m,\beta)$, we have that
\[
\frac{\Lambda_{m,\beta,d-1}}{d}<\beta e \left(1-\frac{1}{2}e^{-\frac{m}{\beta}}\right) \quad (<1).
\]
\end{enumerate}
\end{lemma}

\begin{proof}
The former assertion of $(1)$ follows from $G_{1,\alpha,d-1}(1)=d$ and from $G'_{1,\alpha,d-1}(1)>0$ if $\alpha \in (0,\frac{1}{2})$. The latter assertion of $(1)$ is non-trivial, and showed in \cite[Proposition~4.12]{BCCGNSU}. (There, an expression of $\lim_{d\to \infty}(\frac{1}{d}\Lambda_{1,\alpha,d-1})$ is moreover given.)

Next, we will prove $(2)$. By definition, for every $u\in (0,1)$,
\begin{align*}
\frac{\Lambda_{m,\beta,d-1}}{d}&\leq \frac{1}{d}G_{m,\beta,d-1}(u)\\
&= \frac{1}{d(1-u)}u^{-\beta d}(1-u^{m(d-1)+1}) \\
&<\frac{1}{d(1-u)}u^{-\beta d}(1-u^{md}).
\end{align*}
For $d>\frac{1}{\beta}$, substitute $u=u_{\beta,d}:=1-\frac{1}{\beta d}\in (0,1)$. Then we obtain that
\[
\frac{\Lambda_{m,\beta,d-1}}{d}\leq \beta u_{\beta,d}^{-\beta d}(1-u_{\beta,d}^{md}).
\]
As $d\to \infty$, the right-hand side of the inequality above converges to $\beta e(1-e^{-\frac{m}{\beta}})$. Therefore, if $\beta\in (0,\frac{1}{e}]$, then for sufficiently large $d$ (depending on $m$ and $\beta$), we have that 
\[
\frac{\Lambda_{m,\beta,d-1}}{d}< \beta e \left(1-\frac{1}{2}e^{-\frac{m}{\beta}}\right) \quad (<1).
\]
\end{proof}

\section{Proof of Theorem~\ref{mtheorem=slicerank}}\label{section=slicerank}

First we prove the following proposition.

\begin{proposition}[Expression of the upper bound in Theorem~\ref{mtheorem=slicerank}]\label{proposition=slicerank}
Let $p$ be a prime. Consider a balanced and irreducible system \eqref{T} of $L$ $\mathbb{F}_p$-equations in  $r$ variables:
\begin{equation*}
\left\{
\begin{aligned}
a_1^{(1)}x_1+a_2^{(1)}x_2+a_3^{(1)}x_3+\cdots +a_r^{(1)}x_r&=0,\\
a_1^{(2)}x_1+a_2^{(2)}x_2+a_3^{(2)}x_3+\cdots +a_r^{(2)}x_r&=0,\\
\vdots\qquad\qquad \qquad\\
a_1^{(L)}x_1+a_2^{(L)}x_2+a_3^{(L)}x_3+\cdots +a_r^{(L)}x_r&=0.\end{aligned}
\right.
\end{equation*}
Let $m_1,m_2,\ldots ,m_r$ be, respectively, the multiplicities of $x_1,x_2,\ldots ,x_r$. Then for every $r$-tuple of real numbers $(\alpha_1,\alpha_2,\ldots,\alpha_r)$ satisfying $\alpha_1+\alpha_2+\cdots +\alpha_r=L$ and $\alpha_i\geq 0$ for $i\in [r]$, we have the following estimate for every $n\in \mathbb{N}$:
\[
\mathrm{ex}_{\mathcal{T}}(n)\leq \sum_{i=1}^r (\Lambda_{m_i,\alpha_i,p-1})^n.
\]
\end{proposition}

\begin{proof}
Fix an $r$-tuple $(\alpha_1,\alpha_2,\ldots,\alpha_r)$ as above. Assume that 
$A\subseteq \mathbb{F}_p^n$ is strongly $\mathcal{T}$-free. 
Define a polynomial $f\colon A^r\to \mathbb{F}_p$ by
\begin{align*}
f(x_1,\ldots ,x_r):=\prod_{l=1}^L\left(\prod_{k=1}^n\left(1-(a_1^{(l)}x_{1,k}+a_2^{(l)}x_{2,k}+\cdots +a_r^{(l)}x_{r,k})^{p-1}\right)\right).
\end{align*}
Here, for every $i\in [r]$, we express $x_i\in \mathbb{F}_p^n$ as $x_i=(x_{i,1},x_{i,2},\ldots ,x_{i,n})$. We claim that this polynomial $f$ fulfills condition  \eqref{ast} as in Proposition~\ref{proposition=Tao}. Indeed, first note that for $a\in \mathbb{F}_p$, 
\[
a^{p-1}=\left\{\begin{array}{lc}
1, &\textrm{if $a\ne 0$},\\
0, &\textrm{if $a= 0$}.
\end{array}\right.
\]
Then combine it with strong $\mathcal{T}$-freeness of $A$; it follows that 
\[
f(x_1,x_2,\ldots,x_r)=\left\{\begin{array}{lc}
1, &\textrm{if $x_1=x_2=\cdots=x_r$},\\
0, &\textrm{otherwise}.
\end{array}\right.
\]
Therefore, Proposition~\ref{proposition=Tao} implies that $\#A=\mathrm{sr}(f)$. 
Now we proceed to bounding $\mathrm{sr}(f)$ from above. In the definition of $f$, expand the product. Then $f$ may be written as the sum of scalar multiples of monomials of the following form:
\[
x_{1,1}^{\theta_{1,1}}\cdots x_{1,n}^{\theta_{1,n}}x_{2,1}^{\theta_{2,1}}\cdots x_{2,n}^{\theta_{2,n}}\cdots x_{r,1}^{\theta_{r,1}}\cdots x_{r,n}^{\theta_{r,n}}.
\]
Here, for each $i\in [r]$, setting $\theta(i):=\sum_{k=1}^n\theta_{i,k}$, we have that the sequence of non-negative integers $(\theta_{i,k})_{i\in [r],\ k\in [n]}$ satisfies that
\[
\theta_{i,k}\in \{0,1,\ldots ,m_i(p-1)\},\quad \textrm{and}\quad \sum_{i=1}^{r}\theta(i)=L(p-1)n.
\]
Since $\alpha_1+\alpha_2+\cdots +\alpha_r=L$, for each monomial above, there exists at least one $i\in [r]$ that satisfies $\theta(i)\leq \alpha_i(p-1)n$. For each monomial, we choose the smallest index $i\in [r]$ with the property above. For $i\in [r]$, let $f_i$ be the sum of monomials with coefficients whose index, determined in the manner above, equals $i$. (Here we regard the empty sum as $0$, if it appears.) Since $f=\sum_{i=1}^rf_i$, subadditivity of the slice rank implies that $\mathrm{sr}(f)\leq \sum_{i=1}^r\mathrm{sr}(f_i)$.

Note that for each $i\in [r]$, $f_i$ may be written as the sum of the form
\[
f_i(x_1,x_2,\ldots ,x_r)=\sum_{\theta=(\theta_{i,1},\ldots ,\theta_{i,n})\in \Theta_{m_i,\alpha_i,p-1}}f_{\theta}(x_i)g_{\theta}(x_1,\ldots ,x_{i-1},x_{i+1},\ldots ,x_{r}),
\]
where $\Theta_{m,\alpha,h}$ is the set of indices defined in Lemma~\ref{lemma=sum}; Lemma~\ref{lemma=sum} shows that $\mathrm{sr}(f_i)\leq (\Lambda_{m_i,\alpha_i,p-1})^n$. Therefore, 
\[
\#A(=\mathrm{sr}(f))\quad \leq \sum_{i=1}^r(\Lambda_{m_i,\alpha_i,p-1})^n,
\]
as desired.
\end{proof}

\begin{proof}[Proof of Theorem~\ref{mtheorem=slicerank}]
First, given a system $\mathcal{T}$, define
\[
\overline{C}_{\mathcal{T}}(p):=\inf\left\{\max_{i\in [r]}\Lambda_{m_i,\alpha_i,p-1}: \sum_{i=1}^r\alpha_i=L,\ \alpha_i\geq 0\right\}.
\]
Then, Proposition~\ref{proposition=slicerank} implies that $\mathrm{ex}_{\mathcal{T}}(n)\leq r (\overline{C}_{\mathcal{T}}(p))^n$. The direct product construction shows that for $n,n'\in \mathbb{N}$, in general, $\mathrm{ex}_{\mathcal{T}}(n+n')\geq\mathrm{ex}_{\mathcal{T}}(n)\cdot \mathrm{ex}_{\mathcal{T}}(n')$. Hence the inequality above, in fact, shows that
\[
\mathrm{ex}_{\mathcal{T}}(n)\leq (\overline{C}_{\mathcal{T}}(p))^n;
\]
this argument is sometimes called a `\textit{power trick}'. For $m\geq m'$, it holds that $\Lambda_{m,\alpha,h}\geq \Lambda_{m',\alpha,h}$. Hence by replacing $m_i$ with $m_{\mathcal{T}}$ if $m_i\geq 2$, we have that $\overline{C}_{\mathcal{T}}(p)\leq \tilde{C}_{(r_1,r_2,L,m_{\mathcal{T}})}(p)$.

Finally, we make a comparison between $\tilde{C}_{(r_1,r_2,L,m_{\mathcal{T}})}(p)$ and $p$. Now assume that the parameters $(r_1,r_2,L)$ satisfy `inequality \eqref{bigstar}'. Then we claim that there exists a constant $\tilde{c}_{(r_1,r_2,L,m_{\mathcal{T}})}<1$, not depending on $p$, such that for sufficiently large $p$, we have that
\[
\frac{\tilde{C}_{(r_1,r_2,L,m_{\mathcal{T}})}(p)}{p}< \tilde{c}_{(r_1,r_2,L,m_{\mathcal{T}})}.
\]
Indeed, by `inequality \eqref{bigstar}', there exist $\alpha<\frac{1}{2}$ and $\beta< \frac{1}{e}$ with $r_1\alpha+r_2\beta=L$. Fix such a pair $(\alpha,\beta)$. Then, Lemma~\ref{lemma=lessthanp} implies that
\begin{itemize}
    \item for every $p\geq 3$, $\frac{1}{p}\Lambda_{1,\alpha,p-1}\leq \frac{1}{3}\Lambda_{1,\alpha,2}<1$,
  \item for sufficiently large $p$, $\frac{1}{p}\Lambda_{m_{\mathcal{T}},\beta,p-1}<\beta e(<1)$.
\end{itemize}

What remains is treatment of the case where $p$ is small. Here we employ a variant of `\textit{tensor power trick}' on finite fields. For such a small $p$, take $N\in \mathbb{N}$ sufficiently large such that the inequality in the second item above is valid if we replace $p$ with $q:=p^N$. Note that there is an upper bound for such an $N$ because $p\geq 2$. Consider the finite field $\mathbb{F}_{q}$ and consider $\mathcal{T}$ as a system of $\mathbb{F}_q$-equations via standard embedding $\mathbb{F}_p\subseteq \mathbb{F}_q$. We write $\mathcal{T}_q$ for this system. Replace the power $(p-1)$ with the power $(q-1)$ in the construction of the polynomial $f$ in the proof of Proposition~\ref{proposition=slicerank}. Then, the slice rank method works as well for $\mathcal{T}_q$ and it provides that
\[
\mathrm{ex}_{\mathcal{T}_q}(n)\leq (\tilde{C}_{(r_1,r_2,L,m_{\mathcal{T}})}(q))^n.
\]
Now observe that all coefficients of the system $\mathcal{T}_q$ are in $\mathbb{F}_p$ and that as $\mathbb{F}_p$-vector space, $\mathbb{F}_q\simeq (\mathbb{F}_p)^N$. It then follows that $\mathrm{ex}_{\mathcal{T}_q}(n)=\mathrm{ex}_{\mathcal{T}}(nN)(\geq (\mathrm{ex}_{\mathcal{T}}(n))^N)$. Therefore, by setting $c_{\mathcal{T}}(p):=\frac{1}{p} \left(\tilde{C}_{(r_1,r_2,L,m_{\mathcal{T}})}(q)\right)^{\frac{1}{N}}$, we have the assertion even for small $p$.
\end{proof}

We, furthermore, extend Theorem~\ref{mtheorem=slicerank} to the `\textit{multicolored version}'. To state it, we provide the following definition; see \cite[Definition~3.1]{BCCGNSU} for more details.

\begin{definition}[Multicolored $\mathcal{T}$-free sets]\label{definition=multicolored}
Let $\mathcal{T}$ be a balanced system of $\mathbb{F}_p$-equations in $r$ variables. For fixed $X_1,\ldots,X_r\subseteq \mathbb{F}_p^n$ with $\#X_1=\cdots=\#X_r=s \in \mathbb{N}$, we say that an $r$-dimensional perfect matching $M\subseteq X_1\times\cdots\times X_r$ is an \textit{$r$-colored $\mathcal{T}$-free set} if for every $\bm{x}=(x_1,\ldots,x_r)\in  X_1\times\cdots\times X_r$, it holds that
\begin{center}
$\bm{x}$ is a $\mathcal{T}$-semishape\quad  $\Longleftrightarrow$ \quad $\bm{x}\in M$.
\end{center}
\end{definition}
In particular, if $M$ is expressed as $M=\{(x_{1,i},\ldots,x_{r,i}):1\leq i\leq m\}$, then $M$ is $r$-colored $\mathcal{T}$-free if the equivalence
\begin{center}
$(x_{1,i_1},\ldots,x_{r,i_r})$ is a $\mathcal{T}$-semishape\quad  $\Longleftrightarrow$ \quad 
$i_1=\cdots=i_r$
\end{center}
holds true. In this case, each element $(x_{1,i},\ldots,x_{r,i})$ in $M$ can be regarded as an `$r$-colored singleton semishape'. If $X_1=\cdots=X_r$, then $r$-colored $\mathcal{T}$-freeness is equivalent to strong $\mathcal{T}$-freeness defined in Definition~\ref{definition=Erdos}. 

We may obtain the following result. It implies Theorem~\ref{mtheorem=slicerank}  from the reasoning above; nevertheless, we can see it as a corollary to the proofs of (Proposition~\ref{proposition=slicerank} and) Theorem~\ref{mtheorem=slicerank}. Recall from Definition~\ref{definition=Lambda} the definition of $\tilde{C}_{(r_1,r_2,L,m)}(d)$. 

\begin{mcorollary}[Upper bounds for the multicolored version]\label{mcorollary=multicolor}
Let $\mathcal{T}$ be a balanced and irreducible system of $L$ $\mathbb{F}_p$-equations in $r$ variables. Let $(r_1,r_2,L)$ be the parameters of $\mathcal{T}$. Assume that they satisfy `inequality \eqref{bigstar}'. Let $X_1,\ldots,X_r\subseteq \mathbb{F}_p^n$ satisfy $\#X_1=\cdots=\#X_r=s \in\mathbb{N}$. Assume that $M\subseteq X_1\times\cdots\times X_r$ is $r$-colored $\mathcal{T}$-free. Then, we have that 
\[
s\leq (\tilde{C}_{(r_1,r_2,L,m_{\mathcal{T}})}(p))^n.
\]
\end{mcorollary}

\begin{proof}
Consider the following polynomial $F\colon [s]^r\to \mathbb{F}_p$:
\begin{align*}
F(i_1,i_2\ldots ,i_r):=\prod_{l=1}^L\left(\prod_{k=1}^n\left(1-(a_1^{(l)}x_{1,i_1,k}+a_2^{(l)}x_{2,i_2,k}+\cdots +a_r^{(l)}x_{r,i_r,k})^{p-1}\right)\right).
\end{align*}
Here we express $x_{j,i_j}$ as $x_{j,i_j}=(x_{j,i_j,1},x_{j,i_j,2},\ldots ,x_{j,i_j,n})$. Then the slice rank method as in the proof of Proposition~\ref{proposition=slicerank} applies to this $F$; it provides that $s\leq r(\tilde{C}_{(r_1,r_2,L,m_{\mathcal{T}})}(p))^n$. The power trick enables us to drop $r$ in the right-hand side of the inequality above.
\end{proof}

As we mentioned in the Introduction, Corollary~\ref{mcorollary=multicolor} will be employed in the proof of Theorem~\ref{mtheorem=Wshape}; compare with \cite[Theorem~2.2]{Sauermann}.

\begin{remark}\label{remark=lessthanpW}
Here we give more precise estimates of the conclusion of Corollary~\ref{mcorollary=multicolor} in the case where $\mathcal{T}=\mathcal{S}_{\mathrm{W}}(p)$, for we will utilize them in Theorem~\ref{mtheorem=Wshape}. Let $p\geq 3$. Then the parameters of $\mathcal{T}=\mathcal{S}_{\mathrm{W}}(p)$ are $(r_1,r_2,L,m_{\mathcal{T}})=(3,2,2,2)$. Hence, we can set, as stated in Theorem~\ref{mtheorem=Wshape},
\[
C_{\mathrm{W}}(p)=\tilde{C}_{(3,2,2,2)}(p)\quad (=\inf\{\max\{\Lambda_{1,\alpha,p-1},\Lambda_{2,\beta,p-1}\}:3\alpha+2\beta=2,\ \alpha,\beta\geq 0 \}).
\]

First, we will observe for all $p\geq 3$, it holds that $\frac{1}{p}C_{\mathrm{W}}(p)<1$. Let $\alpha_1=\frac{1.9}{3.85}$ and $\beta_1=\frac{1}{3.85}$. Then, $C_{\mathrm{W}}(p)\leq \max\{\Lambda_{1,\alpha_1,p-1},\Lambda_{2,\beta_1,p-1}\}$. Since $\alpha_1<\frac{1}{2}$, Lemma~\ref{lemma=lessthanp}$(1)$ implies that $\frac{1}{p}\Lambda_{1,\alpha_1,p-1}<1$. In what follows, we examine for which $p$, $\Lambda_{2,\beta_1,p-1}<p$ holds. In a similar argument to one in the proof of Lemma~\ref{lemma=lessthanp}$(2)$, we have that for every $p\geq 5$, 
\[
\frac{\Lambda_{2,\beta_1,p-1}}{p}\leq \frac{1}{3.85} \left( 1-\frac{3.85}{p}\right)^{-\frac{p}{3.85}}.
\]
Note that for $x>1$, the function $x\mapsto \left(1-\frac{1}{x}\right)^{-x}$ is decreasing. Since $\left(1-\frac{3.85}{11}\right)^{-\frac{11}{3.85}}=3.42\ldots <3.85$, the inequality above implies that for $p\geq 11$, $\frac{1}{p}C_{\mathrm{W}}(p)<1$. For $p=3$, $5$ and $7$, numerical computations show that $\frac{1}{p}C_{\mathrm{W}}(p)$ are, respectively, at most $0.994$, $0.987$ and  $0.983$. Hence, we conclude that for all prime $p\geq 3$, it holds that $\frac{1}{p}C_{\mathrm{W}}(p)<1$.

Secondly, we will see that for sufficiently large $p$, it holds that $\frac{1}{p}C_{\mathrm{W}}(p)<0.97$. This time we set $\alpha_2=0.428$ and $\beta_2=0.358$. Then, we have that $C_{\mathrm{W}}(p)\leq \max\{\Lambda_{1,\alpha_2,p-1},\Lambda_{2,\beta_2,p-1}\}$.
By setting $u_{2}=1-\frac{0.874964}{p}$ and $v_{2}=1-\frac{2.72792}{p}$, we obtain by numerical computation that
\[
\lim_{p\to \infty}\frac{G_{1,\alpha_2,p-1}(u_2)}{p}=0.969185\ldots,\quad \textrm{and} \quad \lim_{p\to \infty}\frac{G_{2,\beta_2,p-1}(v_2)}{p}=0.969258\ldots.
\]
Therefore, we have that for sufficiently large $p$,
\[
\frac{C_{\mathrm{W}}(p)}{p}< 0.97  \quad (<(0.985)^2 ).
\]
\end{remark}

\section{Dominant reduction and lower bounds of $\mathrm{ex}_{\mathcal{S}}(n,p)$}\label{section=dominant}

In this section, we provide an lower bound of $\mathrm{ex}_{\mathcal{S}}(n,p)$  under certain conditions in terms of the dominant property of $\mathbb{Z}$-equations. Recall from the Introduction that for every $p\geq 3$ and every $n\in \mathbb{N}$, we have $\mathrm{ex}_{\mathcal{S}_{\mathrm{W}}}(n,p)=1$. Here $\mathcal{S}_{\mathrm{W}}$ is the system of $\mathbb{Z}$-equations representing a `$\mathrm{W}$ shape'. Hence, in this case, the upper bound given by Theorem~\ref{mtheorem=slicerank} itself is by no means reasonable.

\subsection{Proof of Theorem~$\ref{theorem=dominant}$}\label{subsection=dominantirreducible}
Recall from Definition~\ref{definition=dominant} the definition for a (balanced) $\mathbb{Z}$-equation $b_1x_1+b_2x_2+\cdots +b_rx_r=0$ to be \textit{dominant}. Note that this notion uses the natural total order $(\mathbb{Z},\geq )$; hence in this section we mainly discuss systems of $\mathbb{Z}$-equations, and later take $\mathrm{mod}$ $p$ reductions to obtain systems of $\mathbb{F}_p$-equations.
\begin{definition}[Standard form of a dominant equation]\label{definition=standard}
Let $b_1x_1+b_2x_2+\cdots +b_rx_r=0$ be a balanced and dominant $\mathbb{Z}$-equation. Then we may rewrite the equation above, only by transposition of terms, as that of the form
\[
b'_jx_j=\sum_{i\in [r]\setminus \{j\}}b_i'x_i,\quad \textrm{where},\quad b'_j=\sum_{i\in [r]\setminus \{j\}}b_i'.
\]
Here $b'_j>0$, for all $i\in [r]$, $b'_i\geq 0$, and $b_j'=\sum_{i\in [r]\setminus \{j\}}b_i'$. It is called a \textit{standard form} of the dominant equation.
\end{definition}
To see the assertion in Definition~\ref{definition=standard}, let $j$ be a  dominant index. Then $b'_j:=b_j$ and for $i\in [r]\setminus \{j\}$, $b'_i:=-b_i$ works if $b_j>0$. If $b_j<0$, then multiply $(-1)$ to all coefficients above. The value $b'_j$ in the standard form above coincides with the dominant coefficient  of the original $\mathbb{Z}$-equation. A standard from is uniquely determined unless the original equation is of the form $bx_i-bx_j=0$, where $b\in \mathbb{N}$ and $\{i,j\}\in \binom{[r]}{2}$.

\begin{example}\label{example=dominant}
\begin{enumerate}[$(1)$]
\item Consider the system \eqref{SW} of $\mathbb{Z}$-equations. The first equation $x_1-x_2-x_3+x_4=0$ is not dominant. The second equation $x_1-2x_3+x_5=0$ is dominant, and the standard form is $2x_3=x_1+x_5$. Its  dominant index is $3$ and dominant coefficient is  $2$. This system $\mathcal{S}_{\mathrm{W}}$  is irreducible but \textit{not} dominant.
  \item Consider the following system of $\mathbb{Z}$-equations in $7$ variables:
\begin{equation*}\label{S2}
\left\{
\begin{aligned}
x_1+x_2+x_3+x_4-4x_5&=0,\\
x_1+x_2-x_5-x_6&=0,\\
x_1-2x_6+x_7&=0.
\end{aligned}
\right.
\tag{$\mathcal{S}_2$}
\end{equation*}
The first and the third equations are both dominant; their dominant coefficients are respectively $4$ and $2$. The second equation is not dominant. However, the subsystem consisting of the first and the third equations:
\begin{equation*}\label{S2'}
\left\{
\begin{aligned}
x_1+x_2+x_3+x_4-4x_5&=0,\\
x_1-2x_6+x_7&=0.
\end{aligned}
\right.
\tag{$\mathcal{S}_2'$}
\end{equation*}
is a dominant and irreducible. (Recall the definition of irreducibility of a subsystem from Section~\ref{section=result}.) The dominant coefficient $b_{\mathcal{S}_2'}$ of the subsystem $\mathcal{S}_2'$ equals $4$.
\end{enumerate}
\end{example}

The proof of Theorem~\ref{theorem=dominant} will go along a similar line to one in the work of (Behrend \cite{Behrend} and) Alon \cite[Lemma 17, Corollary 18]{FoxPham}. The key here is strict convexity of Euclidean norms (strict convexity of the unit sphere in Euclidean spaces). 

\begin{proposition}\label{proposition=Behrend}
Let $k\geq 1$. Consider a balanced system \eqref{S} of $\mathbb{Z}$-equations
\begin{equation*}
\left\{
\begin{aligned}
b_1^{(1)}x_1+b_2^{(1)}x_2+b_3^{(1)}x_3+\cdots +b_r^{(1)}x_r&=0,\\
b_1^{(2)}x_1+b_2^{(2)}x_2+b_3^{(2)}x_3+\cdots +b_r^{(2)}x_r&=0,\\
\vdots\qquad\qquad \qquad\\
b_1^{(L)}x_1+b_2^{(L)}x_2+b_3^{(L)}x_3+\cdots +b_r^{(L)}x_r&=0,
\end{aligned}
\right.
\end{equation*}
that is dominant and irreducible. 

Then for every $n\geq 2$, there exists $Y\subseteq \{0,1,\ldots ,k\}^n(\subseteq \mathbb{Z}^n)$  that fulfills the following two conditions: 
\begin{enumerate}[$(a)$]
   \item $\#Y\geq \frac{(k+1)^n}{nk^2}$.
   \item All solutions $(y_1,y_2,\ldots ,y_r)$ of $\mathcal{S}$ in $Y$ satisfy that $y_1=y_2=\cdots =y_r$. $($In other word, $Y$ is `strongly $\mathcal{S}$-free', if we extend the definition of strong freeness to a system of $\mathbb{Z}$-equations.$)$
\end{enumerate}
\end{proposition}

\begin{proof}
Via $\mathbb{Z}\subseteq \mathbb{R}$, we regard each $y\in \{0,1,\ldots ,k\}^n\setminus \{(0,0,\ldots ,0),(k,k,\ldots ,k)\}$ as an element in $(\mathbb{R}^n,\|\cdot\|_2)$, where $\|\cdot\|_2$ denotes the Euclidean norm. There are $(k+1)^n-2$ elements of this form; the possible values of the square of the Euclidean norms of them lie in $[1,(n-1)k^2+(k-1)^2]\cap \mathbb{Z}$. By pigeon hole principle, there exists $R\in \mathbb{R}_{>0}$ such that at least $\frac{(k+1)^n-2}{(n-1)k^2+(k-1)^2}$ elements in $\{0,1,\ldots ,k\}^n\setminus \{(0,0,\ldots ,0),(k,k,\ldots ,k)\}$ are on the sphere  $\{y \in \mathbb{R}^n: \|y\|_2=R\}$ of radius $R$ centered at the origin. Fix such an $R$, and set $Y=\{y\in \{0,1,\ldots ,k\}^n : \|y\|_2=R\}$. 

We claim that this $Y$ fulfills the two conditions $(a)$ and $(b)$ above. For $(a)$, by construction,
\begin{align*}
\#Y&\geq \frac{(k+1)^n-2}{(n-1)k^2+(k-1)^2} \\
&\geq \frac{(k+1)^n}{nk^2},
\end{align*}
for $k\geq 1$ and $n\geq 2$. 

In what follows, we confirm $(b)$. We consider a single $\mathbb{Z}$-equation in $\mathcal{S}$; it is of the form $b_1^{(l)}x_1+b_2^{(l)}x_2+b_3^{(l)}x_3+\cdots +b_r^{(l)}x_r=0$ with $l\in [L]$. By assumption, this equation is dominant. Let $j(l)\in [r]$ be a dominant index of it. Then by standard form, a solution $(y_1,y_2,\ldots ,y_r)$ satisfies that
\[
b^{(l)}_{j(l)}{}'y_{j(l)}=\sum_{i\in [r]\setminus \{j(l)\}} b^{(l)}_{i}{}'y_{i}.
\]
By dividing the equality above by $b^{(l)}_{j(l)}{}'R(>0)$, we have that
\[
\frac{y_{j(l)}}{R}=\sum_{i\in [r]\setminus \{j(l)\}} \frac{b^{(l)}_{i}{}'}{b^{(l)}_{j(l)}{}'}\cdot \frac{y_{i}}{R}.
\]
We claim that for every $i\in e_{l}$, it holds that $y_i=y_{j(l)}$. Here $e_l$ is the edge indexed by $l$ in the hypergraph $\mathcal{H}_{\mathcal{S}}$, defined in Definition~\ref{definition=multiplicity}. Indeed, if there exists $i\in [r]\setminus \{j(l)\}$ such that $b^{(l)}_{i}{}'=b^{(l)}_{j(l)}{}'$, then the claim above follows because for all $h \in [r]\setminus \{i,j(l)\}$, $b^{(l)}_{h}{}'=0$. Otherwise, observe that each $\frac{y_i}{R}$, $i\in [r]$, is on the unit sphere and that the right-hand side of the equality above is a convex combination of such vectors with non-trivial coefficients. By strict convexity of the Euclidean norm $\|\cdot\|_2$, we then conclude for every $i\in e_{l}$ that $y_i=y_{j(l)}$. See \cite{Cioranescu} for more details on convexity of norms of real normed spaces.

Since $\mathcal{S}$ is assumed to be irreducible, it follows that $y_1=y_2=\cdots =y_r$. Indeed, for every $\{i_1,i_2\}\in \binom{[r]}{2}$, we can construct a sequence of vertices that connects $i_1$  and $i_2$ in $\mathcal{H}_{\mathcal{S}}$; see Definition~\ref{definition=multiplicity}$(2)$. Therefore, $Y\subseteq \{0,1,\ldots ,k\}^n$ also satisfies $(b)$.
\end{proof}

\begin{proof}[Proof of Theorem~$\ref{theorem=dominant}$]
Given a system $\mathcal{S}$ of $\mathbb{Z}$-equations, take a subsystem $\mathcal{S}'$ that is dominant and irreducible by assumption. Then it holds that $\mathrm{ex}_{\mathcal{S}}(n,p)\geq \mathrm{ex}_{\mathcal{S}'}(n,p)$, 
where $\mathcal{S}'$ is viewed as a system of equations in the original $r$ variables. Hence, we may assume $\mathcal{S}'=\mathcal{S}$ to prove Theorem~\ref{theorem=dominant}; we simply write $b_{\mathcal{S}}$ for $b_{\mathcal{S}'}$.

Fix a prime $p$ with $p>b_{\mathcal{S}}$. Let $k=\left\lfloor \frac{p-1}{b_{\mathcal{S}}} \right\rfloor $. Recall that we define $\mathcal{S}(p)$ as the system  of $\mathbb{F}_p$-equations obtained by $\mathrm{mod}$ $p$ reduction of $\mathcal{S}$. We claim the following: There is a one-to-one correspondence between $\mathcal{S}(p)$-semishapes in $X\subseteq \{0,1,\ldots ,k\}(\subseteq \mathbb{F}_p^n)$ and solutions of $\mathcal{S}$ in $X\subseteq \{0,1,\ldots ,k\}\subseteq \mathbb{Z}^n$,
\[
(\mathbb{F}_p^n\supseteq)\{0,1,\ldots ,k\}^n\ni (x_1,x_2,\ldots ,x_r)\quad \longleftrightarrow \quad (y_1,y_2,\ldots ,y_r) \in \{0,1,\ldots ,k\}^n (\subseteq \mathbb{Z}^n).
\]
Indeed, observe that $kb_{\mathcal{S}}\leq p-1$. In general, there will be degeneration of information by the process of $\mathrm{mod}$ $p$ reduction. However, by viewing dominant $\mathbb{Z}$-equations as their standard forms, we see that this degeneracy does not occur in the setting above.

Hence by $\mathrm{mod}$ $p$ reduction of the construction of $Y$ in Proposition~\ref{proposition=Behrend}, we have that 
\[
\mathrm{ex}_{\mathcal{S}}(n,p)\geq \frac{(k+1)^n}{nk^2}. 
\]
Note that for every $\epsilon>0$, there exists $n_0'=n_0'(k,\epsilon)$ such that for all $n\geq n_0'$, it holds that 
\[
\frac{(k+1)^n}{nk^2}\geq \{(1-\epsilon)(k+1)\}^{n}. 
\]
It completes the proof of the former assertion of Theorem~\ref{theorem=dominant}. Here, note that since $k$ depends on $p$, $n_0'(k,\epsilon)$ can be seen as $n_0(p,\epsilon)$.

Finally, we prove the latter assertion of Theorem~\ref{theorem=dominant}. Since $p> b_{\mathcal{S}}\geq 2$ and  since $p$ is prime, $\frac{p}{b_{\mathcal{S}}}\not \in \mathbb{Z}$. It implies that
\[
k+1>\frac{p}{b_{\mathcal{S}}}.
\]
Hence, by choosing $\epsilon=\epsilon(p)$ sufficiently small depending on $p$ we have that 
\[
(1-\epsilon(p))(k+1)>\frac{p}{b_{\mathcal{S}}}. 
\]
Then, the former assertion implies the latter for $n$ sufficiently large depending on $p$. More precisely, for the $\epsilon (p)$ above, every $n$ with $n\geq n_0(p,\epsilon(p))$ works.
\end{proof}

Concerning lower bounds of $\mathrm{ex}^{\sharp}_{\mathcal{S}}(n,p)$, we have the following.

\begin{theorem}[Lower bounds of $\mathrm{ex}^{\sharp}_{\mathcal{S}}(n,p)$]\label{theorem=dominantsharp}
Let $\mathcal{S}$ be a balanced system of  $\mathbb{Z}$-equations. Assume that there exists an equation belonging to $\mathcal{S}$ that is dominant; let $b$ be the dominant coefficient of it. If $b\geq 2$, then for every prime $p$ with $p>b$ the following holds true: There exists $n_0=n_0(p)$ such that for every $n\geq n_0(p)$, we have
\[
\mathrm{ex}^{\sharp}_{\mathcal{S}}(n,p)\geq \left(\frac{p}{b}\right)^n.
\]
\end{theorem}

\begin{proof}
Let $r'$ be the number of non-zero coefficients in the dominant equation of $\mathcal{S}$. Let $\mathcal{S}''$ be a system of equation(s), viewed as one in $r'$ variables consisting of the dominant equation only. Then, we have $\mathrm{ex}^{\sharp}_{\mathcal{S}}(n,p)\geq \mathrm{ex}^{\sharp}_{\mathcal{S}''}(n,p)$. By construction, $\mathcal{S}''$ is dominant and irreducible. The latter assertion of Theorem~\ref{theorem=dominant} ends our proof. 
\end{proof}

\begin{remark}\label{remark=lowerW}
Consider the system \eqref{SW} of $\mathbb{Z}$-equations. As we argued in the Introduction, $\mathrm{ex}_{\mathcal{S}_{\mathrm{W}}}(n,p)=1$. On the other hand, for $p\geq 3$ and for sufficiently large $n$, depending on $p$, we have that
\[
\mathrm{ex}^{\sharp}_{\mathcal{S}_{\mathrm{W}}}(n,p)\geq \left(\frac{p}{2}\right)^n.
\]
Indeed, the equation $x_1-2x_3+x_5=0$ belonging to $\mathcal{S}_{\mathrm{W}}$ is dominant, and its dominant coefficient is $2$. The lower bound above now follows from Theorem~\ref{theorem=dominantsharp}.
\end{remark}

\subsection{Dominant reduction}
In this subsection, we introduce a `reduction' operation on a system of $\mathbb{Z}$-equations. This process is stated in terms of the dominant property. In this paper, we call the process a \textit{dominant reduction}. First we describe it by supplying an example, and then we will explain the formulation in a general setting. 

Consider the following example of a system of $\mathbb{Z}$-equations in $6$ variables
\begin{equation*}\label{S3}
\left\{
\begin{aligned}
x_1-x_2-x_3+x_4&=0,\\
x_2-x_3-x_4+x_5&=0,\\
x_1-2x_2+x_6&=0.
\end{aligned}
\right.
\tag{$\mathcal{S}_3$}
\end{equation*}
Since there exists no subsystem of $\mathcal{S}_3$ that is dominant and irreducible, Theorem~\ref{theorem=dominant} does not directly apply to $\mathcal{S}_3$. In what follows, we will observe that nevertheless, a certain argument employing strict convexity, similar to one in Proposition~\ref{proposition=Behrend}, works for this system $\mathcal{S}_3$.

\begin{lemma}[Example of dominant reductions]\label{lemma=dominantreduction}
Assume that $Y\subseteq \mathbb{Z}^n$ is included in some sphere in $(\mathbb{R}^n,\|\cdot\|_2)$ centered at the origin. Then if $(x_1,x_2,\ldots ,x_6)$ is a solution of \eqref{S3} in $Y$, then $x_1=x_2=\cdots=x_6$.
\end{lemma}

\begin{proof}
First, observe that the third equation in $\mathcal{S}_3$, $x_1-2x_2+x_6=0$, is dominant. In an argument utilizing strict convexity, we have that in $Y$, $x_1-2x_2+x_6=0$ is equivalent to the equations $x_1=x_2=x_6$. Hence, if we view $\mathcal{S}_3$ as a system of equations \textit{on $Y$}, it is equivalent to the following system of equations:
\begin{equation*}
\left\{
\begin{aligned}
x_1-x_2-x_3+x_4=0,\\
x_2-x_3-x_4+x_5=0,\\
x_1=x_2=x_6.
\end{aligned}
\right.
\end{equation*}
This observation is the key to the present proof.

In the system above, we may `reduce' three variables $x_1,x_2,x_6$ to one variable; we write $x_{\{1,2,6\}}$ for it. Then, on $Y$, the original system $\mathcal{S}$ may be reduced to the following system of two equations in $4$ variables $x_{\{1,2,6\}},x_3,x_4,x_5$:
\begin{equation*}\label{S2(1)}
\left\{
\begin{aligned}
-x_3+x_4&=0,\\
x_{\{1,2,6\}}-x_3-x_4+x_5&=0.\\
\end{aligned}
\right.
\tag{$\mathcal{S}_3^{(1)}$}
\end{equation*}
Now, in $\mathcal{S}_3^{(1)}$, the first equation $-x_3+x_4=0$ is dominant; we have that $x_3=x_4$. This time, reduce $x_3$ and $x_4$ to one variable $x_{\{3,4\}}$. Then, $\mathcal{S}_3^{(1)}$ is reduced to the following (system of) equation in $3$ variables $x_{\{1,2,6\}},x_{\{3,4\}},x_5$:
\[\label{S2(2)}
x_{\{1,2,6\}}-2x_{\{3,4\}}+x_5=0. 
\tag{$\mathcal{S}_3^{(2)}$}
\]
Unlike $\mathcal{S}_3$ and $\mathcal{S}_3^{(1)}$, this system $\mathcal{S}_3^{(2)}$ is dominant and irreducible as that of an equation in $3$ variables $x_{\{1,2,6\}},x_{\{3,4\}},x_5$. Therefore, Proposition~\ref{proposition=Behrend} applies to $\mathcal{S}_3^{(2)}$; we obtain that $x_{\{1,2,6\}}=x_{\{3,4\}}=x_5$. Since we reduce variables as $x_{\{1,2,6\}}=x_1=x_2=x_6$ and $x_{\{3,4\}}=x_3=x_4$, on $Y$, we have that $x_1=x_2=x_3=x_4=x_5=x_6$.
\end{proof}
By Lemma~\ref{lemma=dominantreduction}, we may have the following lower bound of $\mathrm{ex}_{\mathcal{S}_3}(n,p)$ in a similar argument to one in the proof of Theorem~\ref{theorem=dominant}: For prime $p\geq 3$ and for sufficiently large $n$ depending on $p$, it holds that 
\[
\mathrm{ex}_{\mathcal{S}_3}(n,p)\geq \left(\frac{p}{2}\right)^n.
\]

To formalize the argument in the proof of Lemma~\ref{lemma=dominantreduction}, we define the operation of a \textit{dominant reduction} as follows. Recall from Definition~\ref{definition=multiplicity} that for the hypergraph $\mathcal{H}_{\mathcal{S}}=(J_{\mathcal{S}},H_{\mathcal{S}})$, we have the concept of connected components of $J_{\mathcal{S}}$ in the standard sense. See also Example~\ref{example=irreducible}.

\begin{definition}[Dominant reduction]\label{definition=dominantreduction}
Let $\mathcal{S}$ be a balanced system of $\mathbb{Z}$-equations in $r$ variables $x_1,\ldots, x_r$. Assume that $\mathcal{S}$ admits a subsystem $\mathcal{S}'$ that is dominant. Then we define an operation of a \textit{dominant reduction} as follows, which reduce the system $\mathcal{S}$ in $r$ variables to a new system $\mathcal{S}^{(1)}=\mathcal{S}^{(1)}_{\mathcal{S}'}$ of $\mathbb{Z}$-equations in strictly fewer numbers of variables.
\begin{itemize}
  \item $[$Reduction of variables$]$ We define the variables of $\mathcal{S}^{(1)}$ in the following manner: Let $\mathcal{H}_{\mathcal{S}'}=(J_{\mathcal{S}'},H_{\mathcal{S}'})$ be the hypergraph associated with $\mathcal{S}'$. If $i\in [r]\setminus J_{\mathcal{S}'}$, then we keep the variable $x_i$ in the new system. For $i\in [r]$ with $i\in J_{\mathcal{S}'}$, we reduce all variables $x_j$, where $j$ and $i$ belong to the same connected component, to a single variable $x_K$. Here $K$ is the connected component of $J_{\mathcal{S}'}$ in $\mathcal{H}_{\mathcal{S}'}$ that contains $i$. 

In this process, the number of variables reduce from $r$ by $\#J_{\mathcal{S}'}- c(\mathcal{H}_{\mathcal{S}'})$. Here, $c(\mathcal{H}_{\mathcal{S}'})$ denotes the number of connected components in $\mathcal{H}_{\mathcal{S}'}$.
  \item  $[$Reduction of equations$]$ We define the equations in the system $\mathcal{S}^{(1)}$ as follows: For each equation belonging to $\mathcal{S}$, we keep the term in the variable $x_i$ if $i\in [r]\setminus J_{\mathcal{S}'}$. For each connected component $K$ of $J_{\mathcal{S}'}$ in $\mathcal{H}_{\mathcal{S}'}$, we replace $x_i$ for every $i\in K$ with the new variable $x_K$. If we obtain the trivial equation `$0=0$' after this reduction process, then we remove that equation from the new system.
\end{itemize}

We write $\mathcal{S}\stackrel{\mathrm{d}}{\rightsquigarrow}_{\mathcal{S}'} \mathcal{S}^{(1)}$ if $\mathcal{S}^{(1)}$ is obtained by dominant reduction of $\mathcal{S}$ with respect to $\mathcal{S}'$. If we do not specify $\mathcal{S}'$, then we simply write $\mathcal{S}\stackrel{\mathrm{d}}{\rightsquigarrow} \mathcal{S}^{(1)}$.
\end{definition}

If $\mathcal{S}$ admits a subsystem $\mathcal{S}'$ that is dominant and irreducible, then the resulting new system after the dominant reduction with respect to $\mathcal{S}'$ is $\emptyset(1)$, which denotes the empty system of equations \textit{in one variable}.

The following theorem is a refinement of Theorem~\ref{theorem=dominant}. We may prove it by generalizing the arguments of the proofs of Proposition~\ref{proposition=Behrend} and Theorem~\ref{theorem=dominant}; compare with the proof of Lemma~\ref{lemma=dominantreduction}.

\begin{mtheorem}[Dominant reduction and lower bounds of $\mathrm{ex}_{\mathcal{S}}(n,p)$]\label{mtheorem=dominant}

Let $\mathcal{S}$ be a dominant irreducible system of $\mathbb{Z}$-equations. Assume that $\mathcal{S}$ admits a sequence of dominant reductions that terminates with $\emptyset(1)$:
\[
\mathcal{S}=\mathcal{S}^{(0)}\stackrel{\mathrm{d}}{\rightsquigarrow} \mathcal{S}^{(1)}\stackrel{\mathrm{d}}{\rightsquigarrow}\mathcal{S}^{(2)}\stackrel{\mathrm{d}}{\rightsquigarrow}\cdots \stackrel{\mathrm{d}}{\rightsquigarrow}\mathcal{S}^{(k)}\stackrel{\mathrm{d}}{\rightsquigarrow}\emptyset(1).
\]
Let $\tilde{b}$ be the maximum of the dominant coefficients of the dominant subsystems, which are used in the sequence above of dominant reductions. 

Then for every prime $p> \tilde{b}$, we have the following: For every $\epsilon>0$, there exists $n_0=n_0(p,\epsilon)$ such that for every $n\geq n_0(p,\epsilon)$, it holds that
\[
\mathrm{ex}_{\mathcal{S}}(n,p)\geq \left((1-\epsilon)\left\lfloor \frac{p+\tilde{b}-1}{\tilde{b}}\right\rfloor\right)^n.
\]
\end{mtheorem}

Theorem~\ref{theorem=dominant} is a special case of Theorem~\ref{mtheorem=dominant} where $k=0$.

\section{Proof of Theorem~\ref{mtheorem=Wshape}}\label{section=Wshape}
As we emphasized in the Introduction, in general, it is much more difficult to bound $\mathrm{ex}^{\sharp}_{\mathcal{T}}(n)$ from above than to bound $\mathrm{ex}_{\mathcal{T}}(n)$ from above. Very recently, Sauermann \cite{Sauermann} has introduced a new method to address a problem of this type; she has obtained the following result for the system of $\mathbb{F}_p$-equations in $p$ variables:
\[\label{Tpzerosum}
x_1+x_2+\cdots +x_p=0.\tag{$\mathcal{T}_{p,\mathrm{zerosum}}$}
\]

\begin{theorem}[\cite{Sauermann}]\label{theorem=Sauermann}
For $n\in \mathbb{N}$, let $P(n)$ denotes the number of partitions of $n$. Then, for every prime $p\geq 5$, it holds that
\begin{align*}
\mathrm{ex}^{\sharp}_{\mathcal{T}_{p,\mathrm{zerosum}}}(n)\leq 2p^2P(p)\left(\sqrt{\Lambda_{1,\frac{1}{p},p-1}\cdot p}\right)^{n} \quad (\leq 2p^2P(p)(2\sqrt{p})^n).
\end{align*}
\end{theorem}

Her motivation of Theorem~\ref{theorem=Sauermann} is to provide a new upper bound of the  \textit{Erd\H{o}s--Ginzburg--Ziv constant} $\mathfrak{s}(\mathbb{F}_p^n)$; Sauermann \cite[Corollary~1.2]{Sauermann} deduced the best known bound of it, without the assumption of `property $D$', from Theorem~\ref{theorem=Sauermann}. See \cite{FoxSauermann} for historical developments of the known bounds of the Erd\H{o}s--Ginzburg--Ziv constants.

Sauermann's method may be extended to obtain an upper bound of $\mathrm{ex}^{\sharp}_{\mathcal{T}}(n)$ if $\mathcal{T}$ is a (system of a) single $\mathbb{F}_p$-equation. At the present, there may be a gap to extend it for a general system $\mathcal{T}$; the key lemma \cite[Lemma~3.1]{Sauermann} to her argument will not be modified in the full generality. Compare with Problem~\ref{problem=Sauermann}. In this section, we will see that for $\mathcal{S}_{\mathrm{W}}(p)$ (`$\mathrm{W}$ shapes' with coefficients in $\mathbb{F}_p$), this modification is possible. We will demonstrate Theorem~\ref{mtheorem=Wshape}; then in Remark~\ref{remark=difference}, we describe new points appearing in adaptation of Sauermann's argument in our setting.

The key difference between our proof of Theorem~\ref{mtheorem=Wshape} and Sauermann's original argument in \cite{Sauermann} is the following: The strategy of our proof is to  \textit{focus on a `certain shape'} in the concerning set $A\subseteq \mathbb{F}_p^n$. Then we divide the proof into the following two cases. If there exists a \textit{disjoint} collection of large size of such `shapes', then we follow Sauermann's argument \cite{Sauermann}; see the definition of the disjointness here below. Otherwise, consider the collection of the maximal size of such `shapes', and remove all points appearing in that collection from $A$. Then, \textit{the resulting subset of $A$ does not admit such `shapes'}; thus, we are now ready to employ a result on an upper bound of the size of \textit{weakly-free} subsets with respect to that `shape'. See Remark~\ref{remark=difference} for backgrounds of this strategy. 

In the proof of Theorem~\ref{mtheorem=Wshape}, the `shape' on which we cast a spotlight is a non-degenerate $3$-AP. For the (system of an) $\mathbb{Z}$-equality
\[\label{3AP}
x_1-2x_2+x_3=0,\tag{$\mathcal{S}_{3\mathrm{AP}}$}
\]
the work of Ellenberg--Gijswijt \cite{EllenbergGijswijt}, mentioned in the Introduction, states that for $p\geq 3$, $\mathrm{ex}_{\mathcal{S}_{3\mathrm{AP}}}(n,p)\leq (\Lambda_{1,\frac{1}{3},p-1})^n$. Since for every $p\geq 3$, strong $\mathcal{S}_{3\mathrm{AP}}(p)$ freeness and weak $\mathcal{S}_{3\mathrm{AP}}(p)$-freeness are equivalent, it follows that
\[
\mathrm{ex}_{\mathcal{S}_{3\mathrm{AP}}}^{\sharp}(n,p)\leq (\Lambda_{1,\frac{1}{3},p-1})^n;
\]
we employ this estimate in the proof of Theorem~\ref{mtheorem=Wshape}.

In this section, we concentrate on the system \eqref{SW} of two equations with $r=5$, 
and the $\mathrm{mod}$ $p$ reduction $\mathcal{S}_{\mathrm{W}}(p)$ of it for $p\geq 3$. In what follows, we call $\mathcal{S}_{\mathrm{W}}(p)$-semishape and $\mathcal{S}_{\mathrm{W}}(p)$-shape, respectively, simply semishapes and shapes. (Semishapes are called `cycles' in \cite{Sauermann} in her setting.)

Let $X_1,X_2,X_3,X_4,X_5\subseteq \mathbb{F}_p^n$. Following \cite{Sauermann},  for $\{i,j\}\in \binom{[5]}{2}$, we define the \textit{$(i,j)$-extendability}.

\begin{definition}[\cite{Sauermann}]\label{definition=extendability}
For $x\in X_i$ and for $y\in X_j$, we say that $(x,y)$ is \textit{$(i,j)$-extendable} in $X_1\times X_2\times X_3\times X_4\times X_5$ if there exists a semishape $(x_1,x_2,x_3,x_4,x_5)$ in $X_1\times X_2\times X_3\times X_4\times X_5$ such that $x_i=x$ and $x_j=y$.

We call such a semishape a semishape in $X_1\times \cdots \times X_5$ \textit{$(i,j)$-extended from $(x,y)$}.
\end{definition}

We say that two semishapes $(x_1,x_2,x_3,x_4,x_5)$ and $(x_1',x_2',x_3',x_4',x_5')$ are \textit{disjoint} if the two sets of points in the shapes are disjoint, namely, $\{x_1,x_2,x_3,x_4,x_5\}\cap \{x_1',x_2',x_3',x_4',x_5'\}=\emptyset$. We define disjointness of $k$-APs for $k\geq 3$ in a similar manner to one above. For a semishape $(x_1,x_2,x_3,x_4,x_5)$ and for $i\in [5]$, we call $x_i$ the \textit{$i$-th term} of the semishape.

\begin{proof}[Proof of Theorem~\ref{mtheorem=Wshape}]
Throughout the proof, fix a prime $p\geq 3$. As we see later, in the proof, we will employ Corollary~\ref{mcorollary=multicolor} for the system $\mathcal{S}_{\mathrm{W}}(p)$ to deal with Case~$2$ below.

The complete list of point-configurations of degenerate ($\mathcal{S}_{\mathrm{W}}(p)$-)semishapes is the following. Here by the \textit{point-configuration} of a semishape, we mean the set of disjoint points, without multiplicities or labellings.
\begin{itemize}
  \item One non-degenerate $4$-AP $(x,y,z,w)$: with labellings, there are two possibilities, $(x_1,x_3,x_2=x_5,x_4)$ and $(x_2,x_1=x_4,x_3,x_5)$.
  \item One non-degenerate $3$-AP $(x,y,z)$: with labellings, there are two possibilities, $(x_1=x_2,x_3=x_4,x_5)$ and $(x_1,x_2=x_3,x_4=x_5)$.
  \item Two distinct points $(x,y)$: with labellings, it is of the form $(x_1=x_3=x_5,x_2=x_4)$.
  \item One point $x_1=x_2=x_3=x_4=x_5$.
\end{itemize}
Note that if there exists a non-degenerate $4$-AP in $A$, then we, in particular, have a non-degenerate $3$-AP in $A$.

Take $A\subseteq \mathbb{F}_p^n$ that is weakly $\mathcal{S}_{\mathrm{W}}(p)$-free. Set $t:=\left\lceil\frac{2}{7} \#A\right\rceil \in \mathbb{N}$. We divide our proof into the following two cases.

\begin{enumerate}
  \item[Case~$1$.] There do not exist $t$ disjoint non-degenerate $3$-APs in $A$.
  \item[Case~$2$.] There \textit{does} exist $t$ disjoint non-degenerate $3$-APs in $A$.
\end{enumerate}

\noindent
\textbf{Case~$1$.} In this case, pick a collection of non-degenerate disjoint $3$-APs in $A$ of the maximal size; remove all points appearing in this collection from $A$. Write $A'$ for the resulting subset of $A$. Observe that we have 
\[
\#A'\geq \#A-3(t-1)\geq \frac{\#A}{7}.
\]
The key to the proof for Case~$1$ is the following: By construction, the $A'$ above is weakly $\mathcal{S}_{3\mathrm{AP}}(p)$-free. Hence, by the aforementioned work  of Ellenberg--Gijswijt on $\mathrm{ex}^{\sharp}_{\mathcal{S}_{\mathrm{3AP}}}(n,p)$, we obtain that
\[
\#A'\leq (\Lambda_{1,\frac{1}{3},p-1})^n.
\]
Therefore, in Case~$1$, we have that
\begin{align*}
\#A&\leq 7(\Lambda_{1,\frac{1}{3},p-1})^n.
\end{align*}

\noindent
\textbf{Case~$2$.} Pick a collection of size $t$ of disjoint non-degenerate $3$-APs in $A$. Observe that for each non-degenerate $3$-AP $(a,a',a'')$ (forming a $3$-AP in this order) in $A$, $(a,a,a',a',a'')$ is a semishape in $A$. Hence, we define the following set of disjoint semishapes in $A$:
\[
M=\{(a,a,a',a',a''): \textrm{ $(a,a',a'')$ is a non-degenerated $3$-AP picked in the collection}\}.
\]
By construction, $\#M=t$ holds. For each $i\in [5]$, define $X_i$ as the set of the $i$-th terms of the semishapes in $M$. By disjointness, we have the following:
\begin{itemize}
   \item $\#X_1=\#X_2=\#X_3=\#X_4=\#X_5=t$,
   \item $X_1=X_2$ and $X_3=X_4$,
   \item Three subsets  $X_1(=X_2)$, $X_3(=X_4)$ and $X_5$ of $A$ are pairwise disjoint.
\end{itemize}
Set $\tilde{X}:=X_1\times X_2\times X_3\times X_4\times X_5$. The following two lemmas, specially Lemma~\ref{lemma=extendability}, are the key to the proof for Case~$2$. 

\begin{lemma}\label{lemma=shape}
Every semishape $(x_1,x_2,x_3,x_4,x_5)$ in $\tilde{X}$ satisfies that $x_1=x_2$ and that $x_3=x_4$.
\end{lemma}

\begin{proof}
Suppose, to the contrary, that a semishape $(x_1,x_2,x_3,x_4,x_5)$ in $\tilde{X}$ satisfies $x_1\ne x_2$. Since $A$ is weakly $\mathcal{S}_{\mathrm{W}}(p)$-free, every semishape in $\tilde{X}$ must be degenerate. By the classification of degenerate shapes, stated at the beginning of the present proof, it must satisfy $x_2=x_3$, $x_2=x_5$, $x_1=x_4$ or $x_1=x_3=x_5$. If $x_2=x_3$, then it contradicts $X_2\cap X_3=\emptyset$. Similarly, we may draw a contradiction in every option above; hence $x_1=x_2$. We also conclude that $x_3=x_4$ in a similar manner.
\end{proof}

Define the subset $B$ of $X_1\times X_3$ by 
\[
B:=\{(x,y)\in X_1\times X_3:\textrm{ $(x,y)$ is $(1,3)$-extendable in $\tilde{X}$}\}.
\]

\begin{lemma}[Key lemma to the proof]\label{lemma=extendability}
For $(x,y)\ne (x',y')$ in $B$, we have that 
\[
x-y\ne x'-y'. 
\]
In particular, it holds that $\#B\leq p^n$.
\end{lemma}

\begin{proof}
First, we prove the former assertion. Suppose, to the contrary, that $x-y=x'-y'$ and $(x,y)\ne (x',y')$ are in $B$. Then it follows that $x\ne x'$ and $y\ne y'$. A semishape in $\tilde{X}$ $(1,3)$-extended from $(x,y)$ is of the from $(x,x,y,y,2y-x)$; similarly that from $(x',y')$ is of the from $(x',x',y',y',2y'-x')$. The key observation here is that then
\[
(x_1,x_2,x_3,x_4,x_5):=(x',x,y',y,2y'-x')\in \tilde{X}
\]
is a semishape in $\tilde{X}$ with $x_1\ne x_2$. It contradicts Lemma~\ref{lemma=shape}. See figure~\ref{fig6} for a geometric meaning of this construction.
\begin{figure}[h]
\begin{center}
\includegraphics[width=6.5cm]{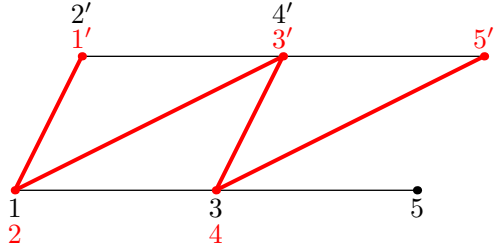}
\caption{Extracting an $\mathcal{S}_{\mathrm{W}}$-semishape}\label{fig6}
\end{center}
\end{figure} 

Secondly, we deduce the latter assertion. It immediately follows because the map $B\ni (x,y)\mapsto x-y\in \mathbb{F}_p^n$ is injective by the former assertion.
\end{proof}

The following lemma enables us to find a subset $M'$ of $M$ of not too small size that is $5$-colored $\mathcal{S}_{\mathrm{W}}(p)$-free; recall Definition~\ref{definition=multicolored}.

\begin{lemma}\label{lemma=color2}
There exists a subset $M'\subseteq M$ of $M$ that satisfies the following:
\begin{itemize}
  \item $\#M'\geq \frac{t^2}{4p^n}$,
  \item $M'$ is $5$-colored $\mathcal{S}_{\mathrm{W}}(p)$-free.
\end{itemize}
\end{lemma}

\begin{proof}
First, we explain the idea, originated in \cite{Sauermann}, to construct $M'$. Find `many' $x\in X_1$ such that for $j\in \{3,4,5\}$, `not too many' $y\in X_j$ satisfies that $(x,y)$ is $(1,j)$-extandable in $\tilde{X}$. By Lemma~\ref{lemma=shape}, a semishape in $\tilde{X}$ is completely determined by the pair $(x_1,x_3)$. Hence, in the idea above, we only need to discuss the case where $j=3$. 

Now we proceed to the precise argument. Set
\[
X_{1,\mathrm{bad}}:=\left\{x\in X_1: \#\{y\in X_3:(x,y)\in B\}\geq \frac{2p^n}{t}\right\}
\]
and $X_{1,\mathrm{good}}:=X_1\setminus X_{1,\mathrm{bad}}$. From the view point of the idea described above, points $x$ in $X_{1,\mathrm{good}}$ may be considered as `good' points. Then, by Lemma~\ref{lemma=extendability}, we have that
\[
\#X_{1,\mathrm{bad}}\leq \frac{p^n}{\frac{2p^n}{t}}\leq \frac{t}{2}.
\]
Hence, we obtain that $\#X_{1,\mathrm{good}}= \#X_1-\#X_{1,\mathrm{bad}}\geq \frac{t}{2}$. Run the following algorithm with initially setting $M':=\emptyset$:
\begin{enumerate}
  \item[(I)] Pick $x\in X_{1,\mathrm{good}}$. Add to $M'$ the semishape in $M$ whose first term is $x$.
  \item[(II)] For the element $x$ chosen in (I), consider all $y\in X_3$ with $(x,y)\in B$. For each such $y$, remove from $X_{1,\mathrm{good}}$ the first term  (if it was still in $X_{1,\mathrm{good}}$) of the semishape  in $M$ whose third term is $y$. 

If  $X_{1,\mathrm{good}}=\emptyset$ after the procedure above, then halt. Otherwise, return to (I).
\end{enumerate}
Note that $x$ itself as in (I) is removed from $X_{1,\mathrm{good}}$ in (II). By construction of the original $X_{1,\mathrm{good}}$, in (II), there are at most $\frac{2p^n}{t}$ points which are removed from $X_{1,\mathrm{good}}$. Therefore, the flow `(I) and (II)' may  be repeated for at least 
\[
\frac{\frac{t}{2}}{\frac{2p^n}{t}}= \frac{t^2}{4p^n}
\]
times. Hence the resulting $M'$ satisfies that $\#M'\geq \frac{t^2}{4p^n}$.

In what follows, we will prove that $M'$ above is $5$-colored $\mathcal{S}_{\mathrm{W}}(p)$-free. Here we change the indices and express $M'$ as 
\[
M'=\{(x_{1,i},x_{2,i},x_{3,i},x_{4,i},x_{5,i}):1\leq i\leq \#M'\}.
\]
For each $j\in[5]$, define $X'_j\subseteq X_j$ as the set of the $j$-th terms of semishapes in $M'$. Set $\tilde{X}':=X_1'\times X_2'\times X_3'\times X_4'\times X_5'$. By construction of $M'$, it follows that $(x_{1,i},x_{3,j})\in X_1'\times X_3'$ is  $(1,3)$-extendable in $\tilde{X}'$ if and only if $i=j$. Suppose now that 
$(x_{1,i_1},x_{2,i_2},x_{3,i_3},x_{4,i_4},x_{5,i_5})$ is a semishape in $\tilde{X}'$. Then Lemma~\ref{lemma=shape} implies that $x_{1,i_1}=x_{2,i_2}$ and $x_{3,i_3}=x_{4,i_4}$. Since $M'$ is a collection of disjoint semishapes, it follows that $i_1=i_2$ and $i_3=i_4$. Moreover, the argument above shows that $i_1=i_3$. Since $x_{1,i_1}-2x_{3,i_3}+x_{5,i_5}=0$,
we have that $x_{5,i_5}=x_{5,i_1}$; by disjointness, it implies that $i_5=i_1$.  Therefore, we conclude that $i_1=i_2=i_3=i_4=i_5$, as desired.
\end{proof}
Apply Corollary~\ref{mcorollary=multicolor}, with respect to the system $\mathcal{S}_{\mathrm{W}}(p)$, to this $M'$. Then by using the constant $C_{\mathrm{W}}(p)$ as in Theorem~\ref{mtheorem=Wshape}, we have that
\[
\left(\frac{t^2}{4p^n}\leq\right)\quad \#M'\leq C_{\mathrm{W}}(p)^n.
\]
Hence, we obtain that $t\leq 2\left(\sqrt{C_{\mathrm{W}}(p)\cdot p}\right)^n$. Recall that $t=\left\lceil\frac{2}{7}\#A\right\rceil$. Therefore, in Case~$2$, we conclude that
\[
\#A\leq  7\left(\sqrt{C_{\mathrm{W}}(p)\cdot p}\right)^n.
\]

Let us unify these two cases to close up our proof. Note that $\Lambda_{1,\frac{1}{3},p-1}< C_{\mathrm{W}}(p)$. Indeed, if $\alpha,\beta\geq 0$ satisfy $3\alpha+2\beta=2$, then $\max\{\alpha,\beta\}\geq \frac{2}{5}>\frac{1}{3}$. (Note also that $\Lambda_{m,\alpha,p-1}$ is strictly increasing both in $m$ and $\alpha$.) Hence, we conclude that in both cases, the inequality
\[
\#A\leq  7 \left(\sqrt{C_{\mathrm{W}}(p)\cdot p}\right)^n
\]
holds true. It complete the proof.
\end{proof}
Recall that in Remark~\ref{remark=lessthanpW}, we show that for every $p\geq 3$, it holds that $C_{\mathrm{W}}(p)<p$. 

\begin{proof}[Proof of Theorem~$\ref{theorem=Wshape}$]
The estimate from above immediately follows from Theorem~\ref{mtheorem=Wshape} and Remark~\ref{remark=lessthanpW}. If $n$ is sufficiently large depending on $p$, the estimate from below is obtained from Remark~\ref{remark=lowerW}.
\end{proof}

\begin{remark}\label{remark=difference}
Here we describe main differences between the argument in the original paper \cite{Sauermann} on $\mathrm{ex}^{\sharp}_{\mathcal{T}_{p,\mathrm{zerosum}}}(n)$ and one in our proof on $\mathrm{ex}^{\sharp}_{\mathcal{S}_{\mathrm{W}}}(n,p)$. More precisely, we explain backgrounds of our strategy to focus on a non-degenerate $3$-AP; recall the argument at the beginning of this section. There is a \textit{gap} for a straightforward application of Sauermann's original argument to nour case, which we describe below. In our case, we may \textit{not} obtain a similar result to Lemma~\ref{lemma=extendability}, which is an adaptation of \cite[Lemma~3.1]{Sauermann}, for a certain type of semishapes. More precisely, this issue occurs if the point-configuration of the semishape is a distinct two points. Namely, if the semishape is of the form $(x_1=x_3=x_5,x_2=x_4)$, where $x_1\ne x_2$. Indeed, let $(x,y,x,y,x)$ and $(x',y',x',y',x')$ be two semishapes obtained by $(1,2)$-extensions. Our task in this case is to construct another semishape $(x_1,x_2,x_3,x_4,x_5)$ out of the two semishapes above that violates $x_1=x_3=x_5$. However, it is \textit{totally impossible}. Indeed, at least two of $x_1$, $x_3$ and  $x_5$ must coincide, then all of them must be the same. Hence, an argument relying on a variant of \cite[Lemma~3.1]{Sauermann} \textit{breaks down} in this case. 

Nevertheless, we have completed our proof of Theorem~\ref{mtheorem=Wshape}; \textit{the trick here} is to put the annoying case above into Case~$1$. Case~$1$ in the proof is treated \textit{all at once by the aforementioned upper bound of Ellenberg--Gijswijt} of non-degenerate $3$-AP-free sets. Thus \textit{we have bypassed a variant of Lemma~$\ref{lemma=extendability}$} for the problematic case above. Furthermore, we obtain the mltiplicative constant part `$7$' in Theorem~\ref{mtheorem=Wshape}, which is good. We do it by putting the case where there exists a collection of large size of disjoint non-degenerate $4$-APs into Case~$2$. If we treat that case separately from Case~$2$, then we would obtain `$15$' instead of `$7$'.

In our short paper \cite{MimuraTokushigeX}, we also study another system of equations
\[\label{astk}
\left\{\begin{aligned}
x_1+x_2-2x_{2k+1}&=0,\\
x_3+x_4-2x_{2k+1}&=0,\\
x_5+x_6-2x_{2k+1}&=0,\\
\vdots\qquad\quad\\
x_{2k-1}+x_{2k}-2x_{2k+1}&=0.
\end{aligned}
\right.
\tag{$\mathcal{S}_{\ast_{k}}$}
\]
and provides an upper bound of $\mathrm{ex}^{\sharp}_{\mathcal{S}_{\ast_k}}(n,p)$. (An $\mathcal{S}_{\ast_k}(p)$-shape consists of  $k$ $3$-APs, sharing the middle term, such that all $2k+1$ terms are distinct; it represents a `$k$-star shape'.) In that case, we focus on an $\mathcal{S}_{\ast_{k-1}}(p)$-shape, instead of a non-degenerate $3$-AP, and argue in an induction on $k\geq 1$.
\end{remark}

By extracting a part of the proof of Theorem~\ref{mtheorem=Wshape}, we obtain an upper bound of `weakly `parallelogram-free' sets. We leave the proof to the reader.

\begin{corollary}[Avoiding a parallelogram]\label{corollary=parallelogram}
Consider the following $\mathbb{Z}$-equation
\[\label{SP}
x_1-x_2-x_3+x_4=0.\tag{$\mathcal{S}_P$}
\]
Then, for every prime $p\geq 3$ and for every $n\in \mathbb{N}$, we have that 
\[
\mathrm{ex}^{\sharp}_{\mathcal{S}_P}(n,p)\leq 7\left(\sqrt{\Lambda_{1,\frac{1}{4},p-1}\cdot p}\right)^n.
\]
\end{corollary}

\section{Further problems}\label{section=problems}

\begin{problem}\label{problem=dominant}
Fix a balanced and irreducible system $\mathcal{S}$ of $\mathbb{Z}$-equations and $($sufficiently large$)$ prime $p$. Is there any procedure to replace $\mathcal{S}$ with another system $\tilde{\mathcal{S}}$ in the same variables such that $\tilde{\mathcal{S}}$ `behaves better' than $\mathcal{S}$ from the aspect of our paper and that $\mathcal{S}(p)$-semishapes coincide with $\tilde{\mathcal{S}}$-semishapes?
\end{problem}

For instance, consider the following system of $\mathbb{Z}$-equations:
\begin{equation*}\label{SPP}
\left\{
\begin{aligned}
x_1-x_2-x_3+x_4&=0,\\
x_2-x_3-x_4+x_5&=0.
\end{aligned}
\right.
\tag{$\mathcal{S}_{PP}$}
\end{equation*}
It represents `two overlapping parallelograms'. We may see that for every $p\geq 3$, $\mathcal{S}_{PP}(p)$-semishapes coincide with $\mathcal{S}_{W}(p)$-semishapes. Hence, for every $p\geq 3$ and $n$, we have that
\[
\mathrm{ex}_{\mathcal{S}_{PP}}^{\sharp}(n,p)=\mathrm{ex}_{\mathcal{S}_{W}}^{\sharp}(n,p).
\]
However, the system $\mathcal{S}_{PP}$ `\textit{behaves worse}' than $\mathcal{S}_{W}$ from the view point of the current paper. For instance, the parameters of $\mathcal{S}_{PP}$ are $(r_1,r_2,L)=(2,3,2)$, and it gives a worse bound if we apply Theorem~\ref{mtheorem=slicerank}, compared with $(r_1,r_2,L)=(3,2,2)$ for the system $\mathcal{S}_{\mathrm{W}}$. Moreover, the system $\mathcal{S}_{PP}$ does not admit any dominant subsystem. Hence we are unable to apply dominant reductions directly to $\mathcal{S}_{PP}$; hence it is, at the first glance, unclear how to obtain a non-trivial lower bound of $\mathrm{ex}_{\mathcal{S}_{PP}}^{\sharp}(n,p)$.

We provide another example. Recall the system \eqref{S3} from Lemma~\ref{lemma=dominantreduction}.  For every $p\geq 3$, the following system gives the same semishapes:
\begin{equation*}
\left\{
\begin{aligned}
x_1-2x_2+x_6&=0,\\
x_1-2x_3+x_5&=0,\\
-2x_4+x_5+x_6&=0.
\end{aligned}
\right.
\end{equation*}
This system itself is dominant and irreducible; we do not need repeat dominant reductions for several times for it. From this aspect, the system above `behaves  better' than $\mathcal{S}_3$.

We note that Problem~\ref{problem=dominant} is visible only after we extend the framework from a single equation to a system of equations.

\begin{problem}\label{problem=Sauermann}
Find more systems $\mathcal{T}$ of $\mathbb{F}_p$-equations for which variants of the argument of Sauermann \cite{Sauermann} applies to obtain non-trivial upper bounds of $\mathrm{ex}^{\sharp}_{\mathcal{T}}(n)$.
\end{problem}

As we discussed in Remark~\ref{remark=difference}, the main difficulty here is extending the key lemma \cite[Lemma~3.1]{Sauermann}. 

In the original paper \cite{Sauermann}, she considers extendability for a pair $(x,y)$. If we can extend the framework to one for a triple $(x,y,z)$ with providing a non-trivial upper bound of $\mathrm{ex}^{\sharp}$, then we may provide broader classes of systems for Problem~\ref{problem=Sauermann}.

\section*{Acknowledgments}
The authors are grateful to the members, Wataru Kai, Akihiro Munemasa, Shin-ichiro, Seki and Kiyoto Yoshino,  of the ongoing seminar on the Green--Tao theorem (on arithmetic progressions in the set of prime numbers) at Tohoku University launched in  October, 2018. Thanks to this seminar, the first-named author has been intrigued with the subject of this paper.  They thank Masayuki Asaoka and Kota Saito for comments. Masato Mimura is supported in part by JSPS KAKENHI Grant Number JP17H04822, and
Norihide Tokushige is supported by JSPS KAKENHI Grant Number JP18K03399.

\bibliographystyle{amsalpha}
\bibliography{W.bib}

\end{document}